\documentclass[11pt]{amsart}

\usepackage{amsmath,amsfonts,amssymb,latexsym,amsthm}
\usepackage{graphicx,epsf}
\usepackage{caption}
\usepackage{subcaption}
\usepackage{epsfig}
\usepackage{psfrag}
\usepackage{tikz}
\usepackage{url}

\newtheorem{thm}{Theorem}[section]
\newtheorem{lemma}[thm]{Lemma}
\newtheorem{mlemma}[thm]{Main Theorem}
\newtheorem{mthm}[thm]{Main Theorem}

\newtheorem{cor}[thm]{Corollary}
\newtheorem{prop}[thm]{Proposition}
\newtheorem{conj}[thm]{Conjecture}

\newtheorem{op}[thm]{Open Problem}

\newtheorem{Example}[thm]{Example}

\def\sq{\square}

\def\aA{\mathbb A}
\def\zz{\mathbb Z}
\def\nn{\mathbb N}
\def\pp{\mathbb P}

\def\rr{\mathbb R}
\def\qqq{\mathbb Q}

\def\sm{\smallsetminus}

\def\Ga{\Gamma}

\def\la{\lambda}
\def\ga{\gamma}

\def\ep{\epsilon}
\def\al{\alpha}
\def\be{\beta}

\def\ve{\varepsilon}

\def\cN{\mathcal D}
\def\cQ{\mathcal Q}
\def\cR{\mathcal R}
\def\CR{\mathcal R}

\def\cf{\mathcal F}
\def\cF{\mathcal F}

\def\cB{\mathcal B}

\def\CT{\mathcal T}

\def\ssu{\subset}

\def\<{\langle}
\def\>{\rangle}

\def\Z{ {\text {\rm Z} } }

\def\scg{\textsc{G}}

\def\Ups{\Upsilon}

\def\.{\hskip.06cm}
\def\ts{\hskip.03cm}
\def\mts{\hspace{-.04cm}}

\def\0{{\mathbf 0}}

\def\nin{\noindent}

\def\Rven{{{\textsc{R}_{n+\varepsilon}}}}
\def\rR{{\textsc{R}}}
\def\CT{\Phi_T}

\def\ba{\mathbf{a}}
\def\bb{\mathbf{b}}
\def\bc{\mathbf{c}}

\def\bz{\mathbf{z}}

\def\pid{L}
\def\BT{B_T}

\def\tweight{{weight }}
\def\weight{\text{\rm weight}}
\def\mult{m}
\def\G{{\textmd{W}}}
\def\U{{\textmd{U}}}
\def\W{\G}
\def\hW{\G_0}

\def\Fpr{F_\circ}
\def\Gpr{G_\circ}
\def\Hpr{H_\circ}
\def\cBB{\mathcal B'}

\begin{document}

\title{Counting with irrational tiles}

\author[Scott Garrabrant]{\ Scott~Garrabrant$^\star$ \ \ }

\author[Igor~Pak]{\ \ Igor~Pak$^\star$}

\date{\today}

\thanks{\thinspace ${\hspace{-.45ex}}^\star$Department of Mathematics,
UCLA, Los Angeles, CA, 90095.
\hskip.06cm
Email:
\hskip.06cm
\texttt{\{coscott,pak\}@math.ucla.edu}}


\begin{abstract}
We introduce and study the number of tilings of unit height
rectangles with irrational tiles.  We prove that the class of
sequences of these numbers coincides with the class of
\emph{diagonals of~$\.\nn$-rational generating functions}
and a class of certain \emph{binomial multisums}.
We then give asymptotic applications and establish
connections to hypergeometric functions and
Catalan numbers. \end{abstract}

\maketitle

\bigskip

\section{Introduction}\label{s:intro}

\noindent
The study of combinatorial objects enumerated by rational generating
functions (GF) is classical and goes back to the foundation of
Combinatorial Theory.  Rather remarkably, this class includes a large
variety of combinatorial objects, from integer points in polytopes and
horizontally convex polyominoes, to magic squares and discordant permutations
(see e.g.~\cite{Sta,FS}).  Counting the number of tilings of a \emph{strip}
(rectangle $[k\times n]$ with a fixed height~$k$) is another popular
example in this class, going back to Golomb, see~\cite[$\S$7]{Gol}
(see also~\cite{BL,CCH,KM,MSV}).  The nature of GFs of such tilings
is by now completely understood (see Theorem~\ref{t:classical} below).

In this paper we present an unusual generalization to tile counting functions
with irrational tiles, of rectangles $[1\times (n + \ve)]$, where $\ve \in \rr$
is fixed.  This class of functions turns out to be very rich and interesting;
our main result (theorems~\ref{t:main} and~\ref{t:main-new} below) is a
complete characterization of these functions.
We then use this result to construct a number of tile counting functions useful
for applications.

\smallskip

Let us first illustrate the notion of tile counting functions with
several examples.  Start with Fibonacci number $F_n$ which count
the number of tilings of $[1\times n]$ with the set~$T$ of
two rectangles $[1\times 1]$ and $[1\times 2]$, see Figure~\ref{f:fib}.

\begin{figure}[hbt]
\begin{center}
\psfrag{1}{}
\psfrag{2}{}
\psfrag{n}{\small $[1\times 10]$}
\psfrag{T}{$T$}
\epsfig{file=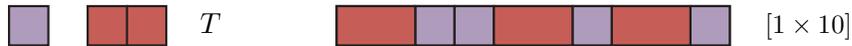,width=10.6cm}
\end{center}
\caption{Fibonacci tiles $T$ and a tiling of $[1\times 10]$.}
\label{f:fib}
\end{figure}

Consider now a more generic set of tiles as in Figure~\ref{f:binom},
where each tile has height~1 and rational side lengths.
Note that the dark shaded tiles here are \emph{bookends}, i.e.\ every tiling
of $[1\times n]$ must begin and end with one, and they are not allowed
to be in the middle.  Also, no reflections or rotations are allowed, only
parallel translations of the tiles.  We then have exactly
$f_T(n)=\binom{n-2}{2}$ tilings of $[1\times n]$, since the two light tiles
must be in this order and can be anywhere in the sequence of $(n-2)$ tiles.

\begin{figure}[hbt]
\begin{center}
\psfrag{m}{\small $[1\times 14]$}
\psfrag{T}{$T$}
\epsfig{file=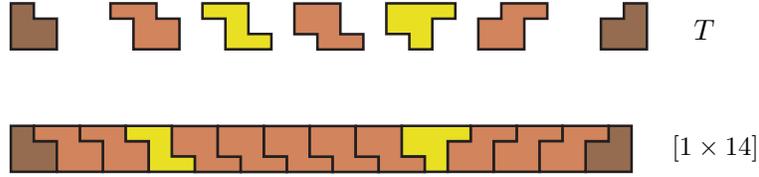,width=9.6cm}
\end{center}
\caption{Set $T$ of $5$ rational tiles and two bookends; a tiling
of $[1\times 14]$ with~$T$.}
\label{f:binom}
\end{figure}

More generally, let $f_T(n)$ be the number of tilings on $[1\times n]$ with
a fixed set of rational tiles of height~1 and two bookends as above.\footnote{For
simplicity, we allow bookends in~$T$ to be \emph{empty tiles}. In general,
bookends play the role of boundary
coloring for Wang tilings~\cite{Wang} \ts (cf.~\cite{GP,PY}). Note that
irrational tilings are agile enough not to require them at all.
This follows from our results, but the reader might enjoy
finding a direct argument.
}
Denote by $\cf_1$ the set of all such functions.  It is easy to see
via the  transfer-matrix method (see e.g.~\cite[$\S$4.7]{Sta}), that
the GF \. $F_T(x) = f(0) + f(1)x+f(2)x^2+\ldots$ \ts is rational:
$$
F_T(x) \. = \, \frac{P(x)}{Q(x)} \ \, \quad \text{for some} \quad P,Q \in \zz[x]\..
$$
In the two examples above, we have GFs \. $1/(1-x-x^2)$ \. and \.
$x^4/(1-x)^3$\ts, \ts respectively.

Note, however, that the combinatorial nature of~$f(n)$ adds further constraints
on possible GFs~$F_T(x)$. The following result gives a complete characterization
of such GFs.  Although never stated in this form, it is well known in a sense
that it follows easily from several existing results (see $\S$\ref{ss:fin-nn}
for references and details).

\begin{thm}\label{t:classical}
Function $f(n)$ is in $\cf_1$, i.e. equal to $f_T(n)$ for all $n\ge 1$ and
some rational set of tiles~$T$ as above,
if and only if its {\rm GF} \. $F(x) = f(0) + f(1)x+f(2)x^2+\ldots$ is $\nn$-rational.
\end{thm}

\smallskip

\noindent
Here the class $\CR_1$  of \emph{$\nn$-rational functions} is defined to be
the smallest class of GFs $G(x)=g(0)+g(1)x+g(2)x^2+\ldots,$
such that:

\smallskip

$(1)$ \ $0, x \in \ts \CR_1$,

$(2)$ \ $G_1, G_2 \in \ts \CR_1$ \ \, $\Longrightarrow$ \ \, $G_1+G_2, G_1\cdot G_2 \in \ts \CR_1$,

$(3)$ \ $G \in \ts \CR_1$, $g(0)=0$ \ \, $\Longrightarrow$ \ \, $1/(1-G) \in \ts \CR_1$\ts.

\smallskip

\noindent
This class of rational GFs is classical and closely related to
\emph{deterministic finite automata} and \emph{regular languages},
fundamental objets in the Theory of Computation (see e.g.~\cite{MM,Sip}),
and Formal Language Theory (see e.g.~\cite{BR,SS}).\footnote{Although
we never state the connection explicitly, both theories give a
motivation for this work, and are helpful in
understanding the proofs (cf.~$\S$\ref{ss:fin-nn}).
}

\smallskip

We are now ready to state the main result.  Let $T$ be a finite set of tiles as above
(no bookends), which all have height~1 but now allowed to have irrational
length intervals in the boundaries.
Denote by $f(n)=f_{T,\ts\ve}(n)$ the number of tilings with $T$ of rectangles
$[1\times (n + \ve)]$, where $\ve \in \rr$ is fixed.  Denote by $\cf$ the set of all
such functions.

Observe that $\cf$ is much larger than $\cf_1$. For example, take $2$ irrational tiles
$\bigl[1\times (\frac12 \pm\al)\bigr]$, for some \ts $\al \notin\qqq$, \ts $0<\al<1/2$,
and let $\ve=0$ (see Figure~\ref{f:half}). Then $f(n) = \binom{2n}{n}$, and the
GF equal to \ts $F(x) = 1/\sqrt{1-4x}$\ts.

\begin{figure}[hbt]
\begin{center}
\psfrag{a}{\footnotesize $\bigl[1\times (\frac12 -\al)\bigr]$}
\psfrag{b}{\footnotesize $\bigl[1\times (\frac12 +\al)\bigr]$}
\psfrag{n}{\small $[1\times 4]$}
\epsfig{file=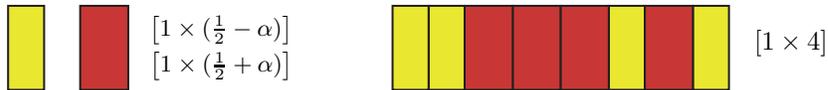,width=10.4cm}
\end{center}
\caption{Set of $2$ irrational tiles; a tiling of $[1\times 4]$ with $8$ tiles.}
\label{f:half}
\end{figure}

Let $\CR_k$ denote the multivariate $\nn$-rational functions defined as a
as the smallest class of GFs \ts $F\in \nn[[x_1,\ldots,x_k]]$, which
satisfies condition

$(1')$ \ $0, \ts x_1,\ldots,x_k \ts \in \ts \CR_1$.

\noindent
and conditions~$(2)$, $(3)$ as above.

\smallskip

\begin{mthm}\label{t:main}
Function $f=f(n)$ is in $\cf$ if and only if
$$f(n) \, = \, \bigl[x_1^{n}\ldots x_k^{n}\bigr]\,
F(x_1,\ldots,x_k) \quad \text{for some}
\quad F\in \CR_k\,.$$
\end{mthm}

\smallskip
\nin
The theorem can be viewed as a
multivariate version of Theorem~\ref{t:classical}, but strictly speaking
it is not a generalization; here the number~$k$ of variables is not specified,
and can in principle be very large even for small~$|T|$
(cf.~$\S$\ref{ss:fin-gessel}).
Again, proving that $\cf$ is a subset of diagonals of rational functions
$F \in \zz[x_1,\ldots,x_k]$ is relatively straightforward
by an appropriate modification of the transfer-matrix method,
while our result is substantially stronger.

\smallskip

\begin{mthm}\label{t:main-new}
Function $f=f(n)$ is in $\cf$ if and only if it can be written as
$$
f(n) \, = \,
   \. \sum_{(v_1,\ldots,v_d)\in\mathbb{Z}^{d}}\,\ts\prod_{i=1}^{r}\.
   \binom{a_{i1} v_1 + \ldots + a_{id}v_d + a'_{i} n + a''_{i}}{b_{i1} v_1 + \ldots + b_{id}v_d  + b'_{i} n + b''_{i}}\,,
$$
for some $r, d\in\mathbb{N}$, and $a_{ij}, b_{ij}, a'_{i}, b'_{i}, a''_{i}, b''_{i}\in \zz$, for all $1\le i\le r$,
$1\le j\le d$.\footnote{The binomial coefficients here are defined to be zero for negative parameters
(see~$\S$\ref{ss:def-binom} for the precise definition); this allows binomial multisums in
the r.h.s.\ to be finite.}
\end{mthm}

\smallskip

The \emph{binomial multisums} (multidimensional sums)
as in the theorem is a special case of a very broad class of
\emph{holonomic functions}~\cite{PWZ}, and a smaller class of
\emph{balanced multisums} defined in~\cite{Gar} (see~$\S$\ref{ss:fin-multisums}).
For examples of binomial multisums, take the
\emph{Delannoy numbers}~$D_n$ (sequence {\tt A001850} in~\cite{OEIS}),
and the \emph{Ap\'{e}ry numbers}~$A_n$ (sequence {\tt A005259} in~\cite{OEIS})~:
$$(\lozenge) \qquad
D_n\, = \, \sum_{k=0}^n \. \binom{n+k}{n-k}\binom{2\ts k}{k}\ts,
\qquad
A_n\, = \, \sum_{k=0}^n \. \sum_{j=0}^k \. \binom{n}{k} \binom{n+k}{k} \binom{k}{j}^3\ts.
$$

\smallskip

In summary, Main Theorems~\ref{t:main} and~\ref{t:main-new}
give two different characterizations of tile counting functions
$f_{T,\ts\ve}(n)$,
for some fixed $\ve\in \rr$ and an irrational set of tiles~$T$.
Theorem~\ref{t:main} is perhaps more structural, while
Theorem~\ref{t:main-new} is easier to use to give explicit constructions
(see Section~\ref{s:three}).  Curiously, neither direction of
either main theorem is particularly easy.

The proof of the main theorems occupies much of the paper.
We also present a number of applications of the main theorems,
most notably to construction of tile counting function with given
asymptotics (Section~\ref{s:app}).  This requires the full power
of both theorems and their proofs.  Specifically, we use the
fact that this class of functions are closed under addition
and multiplication -- this is easy to see for the tile counting
functions and the diagonals, but not for the binomial multisums.

\smallskip

The rest of the paper is structured as follows.  We begin with
definitions and notation (Section~\ref{s:def}).  In the next key
Section~\ref{s:three}, we expand on the definitions of classes $\cF$,
$\cB$ and~$\cR_k$, illustrate them with examples
and restate the main theorems. Then, in Section~\ref{s:app},
we give applications to asymptotics of tile counting functions
and to the Catalan numbers conjecture (Conjecture~\ref{conj:cat}).
In the next four sections~\ref{s:enum}--\ref{s:tech-proofs} we
present the proof of the main theorems, followed by the proofs
of applications (sections~\ref{s:alt} and~\ref{s:tech-apps}).
We conclude with final remarks in Section~\ref{s:fin}.

\bigskip

\section{Definitions and notation}\label{s:def}

\subsection{Basic notation}\label{ss:def-binom}
Let $\mathbb{N}=\{0,1,2,\ldots\}$, $\pp = \{1,2,\ldots\}$, and let
$\aA = \overline{\qqq}$ be the field of algebraic numbers.
For a GF \ts $G\in \zz[[x_1,\ldots,x_k]]$, denote by
\. $\bigl[x_1^{c_1}\ldots x_k^{c_k}\bigr]\,G$ \.
the coefficient of \ts $x_1^{c_1}\ldots x_k^{c_k}$  \ts in~$G$,
and by $[1]\ts G$ the constant term in~$G$.

For sequences $f, g :\nn \to \rr$, we use notation \ts
$f\sim g$ \ts to denote that $f(n)/g(n)\to 1$ as $n\to \infty$.
Here and elsewhere we only use the $n\to \infty$ asymptotics.

We assume that $0!=1$, and $n!=0$ for all $n <0$.  We also
extend binomial coefficients to all $a,b\in\mathbb{Z}$ as follows:
\[
 {a\choose b}=
  \begin{cases}
   \frac{a!}{(a-b)!b!} & \text{\ \ if  \ \ \ } 0\leq b\leq a\ts,\\
   { \quad 1 }        & \text{\ \ if \ \ \ } a=-1,\, \, b=0\ts,\\
   { \quad 0 }      & \text{\ \ otherwise\ts.}
  \end{cases}
\]

\noindent
{\textmd{CAVEAT}:} \ts This is not the way binomial coefficients
are normally extended to negative inputs; this notation
allows us to use ${a+b-1\choose b}$ to denote the number
of ways to distribute $b$ identical objects into $a$
distinct groups, for all $a,b\ge 0$.

\subsection{Tilings}
For the purposes of this paper, a \textit{tile} is an axis-parallel
simply connected (closed) polygon in~$\rr^2$.  A \textit{region}
is a union of finitely many axis-parallel polygons.
We use $|\tau|$ to denote the area of tile~$\tau$.

We consider only finite sets of tiles
$T=\{\tau_1,\ldots,\tau_r\}$.
A \textit{tiling} of a region $\Gamma$ with the set of tiles~$T$,
is a collection of non-overlapping translations of tiles in~$T$
(ignoring boundary intersections), which covers~$\Ga$.
We use $\CT(\Gamma)$ to denote the number of tilings of
$\Gamma$ with~$T$.

A set of tiles $T$ is called \textit{tall} if every tile
in $T$ has height~1.  We study only tilings with tall tiles
of rectangular regions \ts $\rR_a = [1\times a]$, where~$a>0$.

\subsection{Graphs}  \label{ss:def-graphs}
Throughout the paper, we consider finite
\emph{directed weighted multi-graphs}~${\scg}=(V,E)$.  This means
that between every two vertices $v,v'\in V$ there is a finite
number of (directed) edges $v\to v'$, each with its own weight.
A \emph{path}~$\ga$ in~$\scg$ is a sequence of oriented edges
$(v_1,v_2)$, $(v_2,v_3)$, \ldots, $(v_{\ell-1},v_\ell)$;
vertices~$v_1$ and~$v_\ell$ are called start and end of the path.
A \emph{cycle} is a path with $v_1=v_\ell$.  The \emph{\tweight}
of a path or a cycle, denoted $w(\ga)$, is defined to be the
sum of the weights of its edges.

\bigskip

\section{Three classes of functions}\label{s:three}

\subsection{Tile counting functions}  Fix $\varepsilon\ge 0$ and let
$T$ be a set of tall tiles.  In the notation above, $f(n)=\CT(\Rven)$
is the number of tiling of of rectangles $[1\times (n + \ve)]$ with~$T$.
We refer to $f(n)$ as the \textit{tile counting function}.  In notation
of the introduction, $\cf$ is the set of all such functions.
\begin{Example}{\rm \label{ex:list}
We define functions \ts $g_1,\ldots,g_6:\mathbb{N}\rightarrow\mathbb{N}$ \ts as follows:
$$
g_1(n)=
  \begin{cases}
   1 & \text{if } n \text{ is even}\\
   0  & \text{if } n \text{ is odd}
  \end{cases} \, \qquad g_2(n)=2 \,,\qquad g_3(n)=n^2 \,,$$
$$ g_4(n)=2^n \, \qquad g_5(n)=F_n \, \qquad g_6(n)={2n\choose n}\,,$$
where $F_n$ is the $n$-{th} Fibonacci number.  Let us show that these functions are all in $\mathcal{F}$.

\smallskip

First, function $g_1$ counts tilings of a length $n$ rectangle by a single rectangle
$\rR_2$.  Second, consider a set of six tiles $T_2$ as in Figure~\ref{f:ex-tcf},
with dark shaded tiles of area
$\al>0$, $\al \notin \qqq$, the light shaded tiles of area~$1$, and set $\ve = 2\al$.
Now observe that $\Rven$ rectangle can be
tiled with~$T_2$ in exactly two ways: one way using either the first
or the second triple of tiles.

Third, take any two rationally independent irrational numbers $\alpha>\beta>0$,
and set $\varepsilon=\alpha+\beta$.  Consider the set of three rectangles
$T_3 = \{\rR_1, \rR_{1+\al}, \rR_{1+\be},\rR_{1+\al+\be}\}$.  Now observe that there are exactly
$n^2$ tilings of $\rR_{n+\alpha+\beta}$.  Fourth,
take a set $T_4$ with one unit square and two tiles which can only form a unit square,
and observe that $\rR_{n}$ has exactly $2^n$ tilings.  The remaining two examples
are given in the introduction.

\begin{figure}[hbt]
\begin{center}
\psfrag{1}{\small $1$}
\psfrag{2}{\small $2$}
\psfrag{a}{\small $\al$}
\psfrag{1a}{\small \mts \mts $1+\be$}
\psfrag{1b}{\small \mts \mts $1+\al$}
\psfrag{1ab}{\small \mts \mts $1+\al+\be$}
\psfrag{T1}{\small $T_1$}
\psfrag{T2}{\small $T_2$}
\psfrag{T3}{\small $T_3$}
\psfrag{T4}{\small $T_4$}
\epsfig{file=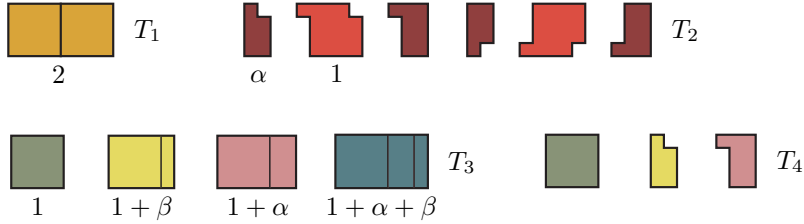,width=10.75cm}
\end{center}
\caption{Tile sets $T_1,\ldots,T_4$ in the example.}
\label{f:ex-tcf}
\end{figure}

}\end{Example}

\subsection{Diagonals of $\mathbb{N}$-rational generating functions}
As in the introduction, let $\mathcal{R}_k$ be the smallest
class of GFs in~$k$ variables $x_1,\ldots,x_k$, satisfying
\begin{enumerate}
\item $0$,~$x_1,\ldots,x_k\in \mathcal{R}_k$\ts,
\item If $F, G\in \mathcal{R}_k$, then $F+F$ and $F\cdot G\in \mathcal{R}_k$.
\item If $F \in \mathcal{R}_k$, and $[1]\ts F=0$, then $\frac{1}{1-F}\in \mathcal{R}_k$.
\end{enumerate}
A GF in $\mathcal{R}_k$ is called an \emph{$\mathbb{N}$-rational
generating function in $k$ variables}. Note that if $G(x_1,\ldots,x_k)\in \cR_k$,
then so is $G(x_1^m,\ldots,x_k^m)$,for all integer~$m\ge 2$.

A \emph{diagonal} of $G\in \nn[[x_1,\ldots,x_k]]$ is a
function $f:\nn \to \nn$ defined by
$$
f(n) \, = \, \bigl[x_1^{n}\ldots x_k^{n}\bigr]\,G(x_1,\ldots,x_k)\..
$$
Denote by $\cN$ the set of
diagonals of all $\mathbb{N}$-rational generating functions,
over all $k\in \pp$.

\begin{Example}{\rm In notation of Example~\ref{ex:list}, let us show that
$g_1,\ldots,g_6\in\cN$~:
$$
g_1(n)=\bigl[x^n]\frac{1}{1-x^2}\.,\ g_2(n) =
\bigl[x^n\bigr]\,\frac{1}{1-x}+\frac{1}{1-x}\,,\ g_3(n)=
\bigl[x^ny^n\bigr]\,x\left(\frac{1}{1-x}\right)^2y\left(\frac{1}{1-y}\right)^2\mts\mts,
$$
\smallskip
$$
g_4(n)=\bigl[x^n\bigr]\,\frac{1}{1-2x}\.,\ g_5(n)=
\bigl[x^n\bigr]\,\frac{1}{1-x-x^2}\.,\ g_6(n)=\bigl[x^ny^n\bigr]\,\frac{1}{1-x-y}\..
$$
\smallskip
}\end{Example}

\subsection{Binomial multisums} \label{ss:three-binom}
Following the statement of Main Theorem~\ref{t:main-new}, denote by $\mathcal{B}$
the set of all functions $f:\mathbb{N}\rightarrow\mathbb{N}$
that can be expressed as
$$
 f(n) \, = \,
   \. \sum_{v\in\mathbb{Z}^{d}}\,\ts\prod_{i=1}^{r}\. {\alpha_{i}(v,n)\choose\beta_{i}(v,n)},
$$
for some $\alpha_{i} = \ba_{i} v + a'_{i} n + a''_{i}$, $\be_{i} = \bb_{i} v + b'_{i} n + b''_{i}$,
where $r, d\in\mathbb{N}$, $\ba_{i}, \bb_{i}: \zz^{d} \to \zz$ are integer linear functions,
and $a'_{i}, b'_{i}, a''_{i}, b''_{i}\in \zz$, for all~$i$.

\smallskip

Note that the summation over all $v\in \zz^{d}$ is infinite, so it is unclear
from the definition whether the multisums $f(n)$ are finite.  However,
the binomial coefficients are zero for the negative values of $\be_{i}$
and $\al_{i}- \be_{i}$, so the summation is in fact over integer points in
a convex polyhedron defined by these inequalities.

\begin{Example}\label{ex:list-binom}
{\rm In notation of Example~\ref{ex:list}, it follows from
the definition that $g_2,g_6 \in\mathcal{B}$.
To see $g_1,g_3,g_4,g_5\in\mathcal{B}$, note that
$$
g_1(n)\. = \. \sum_{v\in\mathbb{Z}} {n\choose 2v}{2v\choose n}, \ \
g_3(n)\. = \. {n\choose 1}{n\choose 1}, \ \
g_4(n) \. =  \. \sum_{v\in\mathbb{Z}}{n\choose v}, \ \
g_5(n)\. =\. \sum_{v\in\mathbb{Z}}{n-v\choose v}.
$$
For the last formula for the Fibonacci numbers $g_5(n)=F_n$ is classical, see e.g.\
\cite[p.~14]{Rio} or~\cite[Exc.~1.37]{Sta}.
}\end{Example}

\medskip

\subsection{Main theorems restated} Surprisingly, the class~$\cB$ of binomial
multisums as above coincides with both tile counting functions and diagonals
of $\nn$-rational functions, and plays an intermediate role connecting them.

\begin{mlemma}\label{t:main-countR}
\. $\mathcal{F}=\cN=\mathcal{B}.$
\end{mlemma}

The proof of Main Theorem~\ref{t:main-countR} is split into three parts.
Lemmas~\ref{countR1},~\ref{D0} and~\ref{D1} state $\mathcal{F}\subseteq\mathcal{B}$,
$\cN\subseteq\mathcal{F}$ and $\mathcal{B}\subseteq\cN$, respectively.
Each is  proved in a separate section, and together they imply the Main Theorem.

\begin{cor}\label{c:add-mult}
The classes of functions $\mathcal{F}=\cN=\mathcal{B}$ are closed under
addition and (pointwise) multiplication.
\end{cor}

This follows from the Main Theorem~\ref{t:main-countR} and Lemma~\ref{l:D2-multi},
which proves the claim for diagonals $f\in \cN$.

\smallskip

Before we proceed to further applications, let us obtain the following
elementary corollary of the Main Theorem~\ref{t:main-countR}.  Note that each
of these tile counting five functions can be constructed directly via ad hoc
argument in the style of Example~\ref{ex:list}.  We include it as an illustration
of the versatility of the theorem.

\begin{cor}\label{c:posapp}
The following functions $f_1,\ldots,f_5: \mathbb{N}\rightarrow\mathbb{N}$
are tile counting functions:
\renewcommand{\theenumi}{\roman{enumi}}%
\begin{enumerate}
\item $f_1$ has finite support,
\item $f_2$ is periodic,
\item $f_3(n)=a_pn^p+\ldots+a_1n+a_0$, where $a_i\in\mathbb{N}$,
\item $f_4(n)=m^n$, where $m\in\mathbb{N}$,
\item $f_5(n)=m^n-1$, where $m\in\mathbb{N}$, $m\ge 1$\ts.
\end{enumerate}
\end{cor}

\begin{proof}  By the main theorem, it suffices to show that each
function $f_i$ is in~$\cN$.  Clearly, function $f_1$
is the diagonal of a polynomial, so $f_1\in\cN$.

The functions
\[
 f_{k,p}(m)=
  \begin{cases}
   1 & \text{if } m=k\mbox{ mod }p\\
   0       & \text{otherwise}
  \end{cases}
\]
are the diagonals of the generating functions
$\frac{x^k}{1-x^p},$
for all $0\leq k<p$, so are clearly in $\cN$, and~$f_2$ can be expressed
as a sum of these $f_{k,p}$ functions.  Since $\cN$ is closed under
addition, this implies $f_2 \in \cN$.
Similarly, the polynomial $f(n)=1$ and $f(n)=n$ are the diagonals of
$1/(1-x)$ and $x/(1-x)^2$ respectively, and thus in~$\cN$.  Since~$\cN$
is closed under addition and multiplication, we have $f_3 \in \cN$.

The function $f_4$ is the diagonal of $\frac{1}{1-mx}$, and therefore
in~$\cN$.  Similarly, the function $f_5$ satisfies the
recurrence $f_5(n+1)=mf_5(n)+(m-1)$,
and thus the diagonal of the generating function $G(x)$ satisfying
$G=mx\ts G+(m-1)/(1-x)$.  Note that
$$
G(x) \. = \. (m-1)\cdot \frac{1}{(1-x)}\cdot \frac{1}{(1-m\ts x)}\..
$$
Therefore, $G \in\cR_1$, which implies $f_5\in \cN$. \end{proof}

\medskip

\subsection{Two more examples}
Recall that our definition of binomial coefficients is
modified to have $\binom{-1}{0}=1$, see~$\S$\ref{ss:def-binom}.
This normally does not affect any (usual) binomial sums,
e.g.\ the Delannoy and Ap\'{e}ry numbers defined in the
introduction remain unchanged when the summations
in~$(\lozenge)$ are extended to all integers.  Simply put,
whenever $\binom{-1}{0}$ appears there, some other binomial
coefficient in the product is equal to zero.

The proof of Main Theorem~\ref{t:main-countR} is constructed
by creating a large number of auxiliary variables for
the $\nn$-rational functions, and auxiliary indices
for the binomial multisums.  These auxiliary indices
are often constrained to a small range, and $\binom{-1}{0}$
does appear in several cases.

\begin{Example} \label{ex:lucas} {\rm
Denote by $L_n$ the \emph{Lucas numbers} \ts $L_n = L_{n-1} + L_{n-2}$,
where $L_1=1$ and $L_2=3$, see e.g.~\cite[$\S$4.3]{Rio}
(sequence {\tt A000204} in~\cite{OEIS}).
They have a combinatorial interpretation as the number of matchings
in an $n$-cycle, and are closely related to Fibonacci numbers $F_n$~:
$$(\circledcirc) \qquad
L_n \. =  \. F_{n} + F_{n-2} \quad \text{for} \  n\ge 2\ts.
$$
From Corollary~\ref{c:add-mult}, the function
$f(n):= L_n$ is in~$\cF$.  In fact, it is immediate that
$L_n \in \cR_1$~:
$$
L_n \, = \, [x^n] \. \frac{1+x^2}{1-x-x^2}\..
$$
To see directly that Lucas numbers are in~$\cF$, take five tiles
as in Figure~\ref{f:lucas}, with two right bookends,
emulating~$(\circledcirc)$.
On the other hand, finding a binomial sum is less intuitive,
as $\cB$ is not obviously closed under addition.  In fact, we have:
$$
L_n \, = \, \sum_{(k,i) \in \zz^2} \. \binom{n-k-2\ts i}{k} \binom{1}{i}\ts,
$$
where we use~$(\circledcirc)$, the formula for $g_5(n)$ in
Example~\ref{ex:list-binom}, and make~$i$ constrained to~$\{0,1\}$.
Note that we avoid using $\binom{-1}{0}$.
\begin{figure}[hbt]
\begin{center}
%
%
\epsfig{file=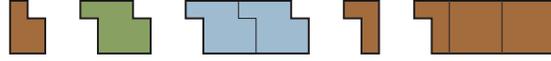,width=7.5cm}
\end{center}
\caption{Five tiles giving Lucas numbers~$L_n$.}
\label{f:lucas}
\end{figure}
}
\end{Example}

\begin{Example}  \label{ex:powers}
{\rm Let $f(n)=2^n+3^n$.  Checking that $f\in \cF$ and
$f\in \cN$ is straightforward and similar to $g_4(n)$
in the examples above.  However, finding a binomial
multisum is more difficult:
$$
f(n) \. = \. \sum_{(i,j,k,\ell,m)\in \zz^5} \binom{n}{i}
\binom{m}{j}\binom{1}{k}\binom{m-k}{m}\binom{\ell+k-1}{\ell}
\binom{i}{m+\ell}\binom{m+\ell}{i}\ts.
$$
Note here that the term $\binom{1}{k}$ gives $k\in\{0, 1\}$.  Also,
$\binom{i}{m+\ell}\binom{m+\ell}{i}$  terms give $m+\ell=i$.
Similarly, $\binom{m-k}{m}\binom{\ell+k-1}{\ell}$ give that
$m=0$ if $k=1$, and $\ell=0$ if $k=0$.  Therefore,
$$
f(n) \, = \, \sum_{(j,m)\in \zz^2} \binom{n}{m}
\binom{m}{j} \. + \. \sum_{\ell\in \zz} \binom{n}{\ell}
\, = \, 2^n \. + \.3^n\ts,
$$
where two sums correspond to the cases $k=0$ and $k=1$, respectively.
Note that $\binom{-1}{0}=1$ is essential in this calculation.
It would be interesting to see if Theorem~\ref{t:main-new}
holds without modification.
}
\end{Example}

\bigskip

\section{Applications}\label{s:app}

\subsection{Balanced multisums}\label{ss:app-balanced}
Define a \emph{positive multisum} to be a function
\ts $g:\mathbb{N}\rightarrow\mathbb{N}$ \ts
that can be expressed as
$$
 g(n) \, = \,
   \. \sum_{v\in\mathbb{Z}^{d}}\,\ts\prod_{i=1}^{r}\. \frac{\al_{i}(v,n)!}{\be_{i}(v,n)! \. \ga_{i}(v,n)!},
$$
for some $\al_{i} = \ba_{i} v + a'_{i} n + a''_{i}$, $\be_{i} = \bb_{i} v + b'_{i} n + b''_{i}$, 
$\ga_{i} = \bc_{i} v + c'_{i} n + c''_{i}$, 
where $r, d\in\mathbb{N}$, $\ba_{i},\bb_i,\bc_i: \zz^{d} \to \zz$ are integer linear functions,
and $a'_{i},\ldots,c''_{i}\in \zz$, for all~$i$.  Here the sum is
over all $v\in\mathbb{Z}^{d}$ for which $\al_i(v,n),\be_i(v,n),\ga_i(v,n)\ge 0$, for all~$i$.

Positive multisum is called \emph{balanced} if \ts $\al_i=\be_i+\ga_i$ for all~$i$.
Denote by $\cBB$ the set of finite sums of balanced positive multisums:
$$
f(n) \. = \. g_1(n) \ts + \. \ldots \. + \ts g_k(n)\ts.
$$

\begin{thm} \label{t:main-cbb}
\ts $\cB=\cBB$.
\end{thm}

The Delannoy and Ap\'{e}ry numbers defined in equation~$(\lozenge)$ 
in the introduction are examples of balanced multisums, as are Lucas numbers,
see Example~\ref{ex:lucas}.  These formulas use only one balanced positive multisum, 
i.e.\ have $k=1$.
However, as Example~\ref{ex:powers} suggests, the sums $f(n)=2^n+3^n$ can we written
with $k=2$, as the lengthy binomial multisum for~$f(n)$ involves using
the $\binom{-1}{0}=1$ notation.
Therefore, one can think of Theorem~\ref{t:main-cbb} as a \emph{tradeoff}:
we prohibit using the $\binom{-1}{0}$ notation, but now allow taking finite sums of
balanced multisums (cf.~$\S$\ref{ss:fin-multisums}).

We give direct proof of the theorem in Section~\ref{s:tech-apps}.
Note that $\cBB$ is trivially closed under addition and multiplication, 
so Theorem~\ref{t:main-cbb} together with the main theorem immediately 
implies Corollary~\ref{c:add-mult}.  

\subsection{Growth of tile counting functions}\label{ss:app-asympt}
We say that a function $f$ is \textit{eventually polynomial} if there exist
an $N\in\mathbb{N}$ and a polynomial $q$ such that for all $n\geq N$,
we have $f(n)=q(n)$.  We say that a function $f$ \ts\textit{grows exponentially},
if there exist $c_1,c_2>0$ and $N\in\mathbb{N}$, such that for all $n\geq N$,
we have $e^{c_1n}\leq f(n)\leq e^{c_2 n}$.

\begin{thm}\label{t:mainR}
Let $f\in \cF$ be a tile counting function. There exists an integer $m\geq 1$,
such that every function $f_i(n):=f(n\ts m+i)$ either grows exponentially
or is eventually polynomial, where $0 \le i \le m-1$.
\end{thm}

In particular, Theorem~\ref{t:mainR} implies that the growth of $f$ is
at most exponential.  Further, if the growth of $f$ is subexponential,
then $f$ must have polynomial growth.  This rules out many natural
combinatorial and number theoretic sequences, e.g.\ the number
of partitions~$p(n)$, or the $n$-th prime~$p_n$, cf.~\cite{FGS}.

The proof of Theorem~\ref{t:mainR} uses the geometry of integer
points in convex polyhedra; it is given in Section~\ref{s:alt}.
The theorem should be contrasted with the following asymptotic
characterization of diagonals of rational functions,
which follows from several known results:

\begin{thm}[See $\S$\ref{ss:fin-asym}]\label{t:growth-DFinite}
Let $f(n)$ be a diagonal of $P/Q$, where
$P,Q\in \zz[x_1,\ldots,x_k]$.  Suppose further that
$f(n) = \exp O(n)$ as $n\to \infty$.
Then there exists an integer~$m\ge 1$,
s.t.
$$
f(n) \, \sim \, A \ts \la^n  \ts n^\al \ts (\log n)^\be\.,
\quad \text{for all} \quad n = i \ \, \text{\rm mod} \,~m\ts,\ \ 0 \le i \le m-1\ts,
$$
where $\al\in \qqq$, $\be \in \nn$, and $\la \in \aA$.
\end{thm}

In our case, the subexponential growth implies $\la=1$,
which gives asymptotics \ts $A \ts n^\al \ts (\log n)^\be$.
Theorem~\ref{t:mainR} implies further that $\al\in \nn$,
$\be=0$, and $A\in \qqq$ in that case.

\begin{Example}{\rm
The following binomial sums show that nontrivial exponents $\al \notin \zz$
and $\be>0$ can indeed appear for $f\in \cF$ and $\la >1$~:
$$
\binom{2n}{n} \, \sim \,  \frac1{\sqrt{\pi}} \.\ts 4^n \ts n^{-1/2}\,,
\qquad \quad
\sum_{k=1}^n \. \binom{2k}{k}^2 \ts 16^{n-k} \, \sim \, \frac{1}{\pi}
\. \ts 16^n \ts \log n\..
$$
}\end{Example}

Following these examples, we conjecture that~$\al$ is always half-integer:

\begin{conj}\label{conj:growth}
Let $f\in \cF$ be a tile counting function.  Then there exists
an integer~$m\ge 1$, s.t.
$$
f(n) \, \sim \, A \ts \la^n  \ts n^\al\ts (\log n)^\be\.,
\quad \text{for all} \quad n = i \ \, \text{\rm mod} \,~m\ts,\ \ 0 \le i \le m-1\ts,
$$

\nin
where $\al\in \zz/2$, $\be \in \nn$, and $\la \in \aA$.
\end{conj}

See $\S$\ref{ss:fin-asym} for a brief overview of related
asymptotic results.

\medskip

\subsection{Catalan numbers}\label{ss:app-cat}
Recall the Catalan numbers:
$$
C_n\, =\, \frac{1}{n+1}{2n\choose n}\.. 
$$
We make the following mesmerizing  conjecture.

\begin{conj}\label{conj:cat}
The Catalan numbers $C_n$ is \emph{not} a tile counting function.
\end{conj}

Several natural approaches to the conjecture can be proved not to work.
First, we show that the naive asymptotic approach cannot be used to prove
Conjecture~\ref{conj:cat}.

\begin{prop}\label{cat4}
For every $\ep>0$, there exists a tile counting function $f\in \cF$, s.t.
$$
f(n) \, \sim \, A\cdot C_n \qquad \text{for some} \quad \ A \in (1-\ep,1+\ep)\..
$$
\end{prop}

In a different direction, we show that Conjecture~\ref{conj:cat} does
not follow from elementary number theory considerations.

\begin{prop}\label{cat6}
For every $m\in \nn$, there exists a tile counting function
$f\in \cF$, s.t.  $f(n)= C_n \mbox{ mod }m$.
\end{prop}

\begin{prop}\label{cat7}
For every prime~$p$, there exists a tile counting function $f\in \cF$,
s.t.  $\text{ord}_p \bigl(f(n)\bigr) =\text{ord}_p (C_n)$, where
$\text{ord}_p(m)= \max\{d \. :  \. p^d|m\}$.
\end{prop}

The results in this subsection are proved in Section~\ref{s:tech-apps}.
See $\S$\ref{ss:fin-ord} for more on the last proposition.

\medskip

\subsection{Hypergeometric functions}\label{ss:app-hype}
We use the following special case of the
\emph{generalized hypergeometric function}:
$$
\ _{p+1}F_{p}(a_1,\ldots,a_p,1;b_1,\ldots,b_p;r) \,\, = \,
\sum_{m=0}^\infty\prod_{k=0}^{m-1}\frac{(k+a_1)(k+a_2)\ldots(k+a_p)\ts r}{(k+b_1)(k+b_2)\ldots(k+b_p)}.
$$

Let $p$ be a positive integer and $\la=(\la_1,\ldots,\la_\ell)\vdash p$ be
a partition of~$p$.
Denote by $\Ups_\la$ the following multiset of $p$ rational numbers:
$$
\Ups_\la \, = \, \bigcup_{i=1}^\ell \. \left\{\frac{1}{\lambda_i},
\frac{2}{\lambda_i},\ldots,\frac{\lambda_i-1}{\lambda_i},1\right\}\ts.
$$
For example, if $\lambda=(5,4,2,1)\vdash 12$, then
$$\Ups_\la \, = \,
\left\{\frac{1}{5},\frac{2}{5},\frac{3}{5},\frac{4}{5},1,\frac{1}{4},
\frac{1}{2},\frac{3}{4},1,\frac{1}{2},1,1\right\}\ts.
$$

\begin{thm}\label{t:hypo}
Let $\mu=(\mu_1,\ldots,\mu_k)\vdash p$, and
let $\nu=(\nu_1,\ldots,\nu_\ell)\vdash p$ be a \emph{refinement}
of~$\mu$.  Write
$$\Ups_\mu = \{a_1,\ldots,a_p\}, \quad \Ups_\nu = \{b_1,\ldots,b_p\}\ts,
$$
and fix $r=r_1/r_2\in\qqq$.  Denote \ts
$A=\ _{p+1}F_{p}(a_1,\ldots,a_p,1;b_1,\ldots,b_p;r)$,
and suppose that $A<\infty$ is well defined.  Finally, let
$c\in\nn$ be a multiple of
all prime factors of $\mu_1\cdot\mu_2\cdot\ldots\cdot\mu_k\cdot r_2$.
Then, there exists a tile counting function $f\in \mathcal{F}$, s.t.
\ts $f(n) \ts \sim \ts A \ts c^n\ts.$
\end{thm}
The proof of Theorem~\ref{t:hypo} is given in Section~\ref{s:tech-apps}.

\begin{cor}
There exists a tile counting function $f\in\mathcal{F}$, such that
$$
f(n)\sim\frac{\sqrt{\pi}}{\Gamma(5/8)\ts \Gamma(7/8)}\, 128^{n}.
$$
\end{cor}

\begin{proof}
Let $p=4$, let $\mu=(4)$, $\nu=(2,1,1)$, and set $r=1/2$.
Then $\Ups_\mu=\{1/4,1/2,3/4,1\}$ and $\Ups_\nu=\{1/2,1,1,1\}$.
Since any even $c$ is allowed, we can take $c=128$.  Then,
by Theorem~\ref{t:hypo}, there exists
$f\in \cF$, s.t.
$$
\frac{f(n)}{128^n} \, \sim \ _5F_{4}\Big(\frac{1}{4},\frac{1}{2},\frac{3}{4},1,1;\frac{1}{2},1,1,1;\frac{1}{2}\Big)
\, =  \ _2F_{1}\Big(\frac{1}{4},\frac{3}{4};1;\frac{1}{2}\Big) \, = \, \frac{\sqrt{\pi}}{\Gamma(5/8)\ts\Gamma(7/8)}\.,
$$
as desired.
\end{proof}

Since the proof of Theorem~\ref{t:hypo} is constructive, we can obtain an
explicit tile counting function~$f(n)$ as in the corollary:
$$f(n)\, = \, \sum_{k=0}^{n} \. {4k \choose k}{3k \choose k} \. 128^{n-k} \,
\sim\, \frac{\sqrt{\pi}}{\Gamma(5/8)\ts\Gamma(7/8)}\, 128^{n}\ts.
$$
The corollary and the theorem suggest that there is no easy characterization
of constants~$A$ in Conjecture~\ref{conj:growth}, at least not enough to
obtain Conjecture~\ref{conj:cat} this way.  Here is yet another quick
variation on the theme.

\begin{cor}\label{c:hypo-transc-6}
There exists a tile counting function $f\in\mathcal{F}$, such that
$$f(n)\, \sim\, \frac{\Gamma(3/4)^3}{\sqrt[3]{2}\,\pi} \, 6^n.$$
\end{cor}
\begin{proof}
Take $p=3$, $\mu=(3)$, $\nu=(1,1,1)$, and proceed as above.
\end{proof}

\bigskip

\section{Tile counting functions are binomial multisums}\label{s:enum}

\nin
In this section, we prove the following result towards the proof of Main Theorem~\ref{t:main-countR}.

\begin{lemma}\label{countR1}
$\mathcal{F}\subseteq\mathcal{B}$.
\end{lemma}

\smallskip

\nin
The proof first restates the lemma in the language of counting cycles in
multi-graphs~$\scg$ (see $\S$\ref{ss:def-graphs}), and then uses graph theoretic
tools to give a binomial-multisum formula for the latter.

\medskip

\subsection{Cycles in graphs}\label{ss:enumR-cycles}
We first show how to compute tile counting functions in the
language of cycles in weighted graph.

\begin{lemma}\label{lemGraph}
For every tile counting function $f(n)$ there exists a finite weighted directed
multi-graph $\scg_T$ with vertices $v_0,\ldots,v_N$, such that $f(n)$ is the number
of paths of \tweight $n+\varepsilon$, which start and end at~$v_0$.
\end{lemma}

\begin{proof}
Fix an $f(n)=\CT(\Rven)$. Recall that each tile $\tau\in T$ has height~1.
Denote by $\partial_L(\tau)$ and $\partial_R(\tau)$ the left boundary and
right boundary curves of~$\tau$ of height~1, respectively.
A sequence of tiles $(\tau_0,\ldots,\tau_\ell)$ is a tiling
of~$\Rven$ if and only if
\begin{enumerate}
  \item $\partial_R(\tau_i)=\partial_L(\tau_{i+1})$ for $0\le i \le \ell-1$\ts,
  \item $\partial_L(\tau_0)$ is a vertical line\ts,
  \item $\partial_R(\tau_\ell)$ is a vertical line\ts,
  \item $|\tau_0|+\ldots+|\tau_\ell|=n+\ve$\ts.
\end{enumerate}
Here the first condition implies that all the tiles fit together with no gaps,
the second and third conditions imply that the union of the tiles is actually
a rectangle, and the fourth condition implies that the rectangle
has length~$n+\ve$.

We now construct a weighted directed multi-graph $\scg_T$ corresponding
to~$T$ as follows. The vertices of $\scg_T$ are exactly the set of left
or right boundaries of tiles (up to translation). Denote them
$v_0,\ldots,v_N$, where $v_0$ is the vertical line.
Let the edge $e_{ij}=(v_i,v_j)$ in $\scg_T$ correspond to tile~$\tau\in T$,
such that $\partial_L(\tau)=v_i$,   $\partial_R(\tau)=v_j$, and
let $\weight(e_{ij})=|\tau|$.  Note that edges $e_{ij}$ and~$e_{ij}'$,
corresponding to tiles $\tau$ and~$\tau'$, can have different weight.
By construction, the paths in $\scg_T$  of \tweight $n+\varepsilon$,
which start and end at~$v_0$, are in bijection with tilings of~$\Rven$\ts.
\end{proof}

\subsection{Irreducible cycles}
To count the number of cycles in a directed
multi-graph starting at~$v_0$ of \tweight $n+\varepsilon$,
we factor the cycles into irreducible cycles.

\smallskip

Let $\scg=(V,E)$ be a finite directed multi-graph, and let
$V=\{v_0,\ldots,v_N\}$.
A cycle~$\ga$ in $\scg$ is called \textit{positive} if it starts and
ends at~$v_i$, and only passes through vertices~$v_j$ with $j\geq i$.
Cycle~$\ga$ is called \textit{irreducible} if it is a positive and
contains no positive shorter cycle~$\ga'$; we refer to~$\ga'$ as
\emph{subcycle} of~$\ga$.

\begin{lemma}\label{l:fin}
There are finitely many irreducible cycles in~$\scg$.
\end{lemma}

\begin{proof} 
We proceed by induction on the number $N+1$ of vertices in~$\scg$.
The claim is trivial for $N=0$.
Suppose $N\ge 1$ and let $\gamma$ be an irreducible cycle in $\scg$.
If $\gamma$ does not contain every vertex in $\scg$, we can delete
an unvisited vertex $v_i$ and apply inductive assumption to
a smaller graph $\scg'=\scg-v_i$.  Thus we can assume that $\ga$
contains all vertices.

Since $\ga$ is positive and contains all vertices,
it must start at~$v_0$.  Since $\ga$ is irreducible,
it never come back to $v_0$ until the end.
Note that $\ga$ passes through $v_1$ exactly one, since
otherwise it is not irreducible.  Identify vertices
$v_0$ and~$v_1$, and denote by~$H$ the resulting
smaller graph.  The cycle $\ga$ is then mapped into a
concatenation of two irreducible cycles in~$H$. Applying
inductive assumption to~$H$ gives the result.
\end{proof}

\subsection{Multiplicities of irreducible cycles}
Let $\rho$ be an irreducible subcycle of a positive
cycle~$\gamma$.  Define $\gamma-\rho$ to be the positive cycle
given by traversing $\gamma$, but skipping over $\rho$.
The \emph{multiplicity} of $\rho$ in $\gamma$,
denoted $\mult(\rho,\gamma)$, is defined to be:
\[
 \rho(\gamma)=
  \begin{cases}
   1 & \text{if } \gamma=\rho,\\
   0 & \text{if } \gamma \text{ is irreducible and not equal to } \rho,\\
   \mult(\rho,\gamma^\prime)+\mult(\rho,\gamma-\gamma^\prime)       &
   \text{if } \gamma^\prime \text{ is an irreducible positive subcycle of } \gamma.
  \end{cases}
\]

\smallskip

\begin{lemma}\label{l:welld}
The multiplicity $\mult(\rho,\gamma)$ is well defined.
\end{lemma}

In other words, the multiplicity $\mult(\rho,\gamma)$
represents the number of times $\rho$ appears in the
decomposition of~$\gamma$.  This allows us to count cycles
in~$\scg_T$, which start and end at~$v_0$, that decompose into
a given list of irreducible cycles.

\begin{proof}[Proof of Lemma~\ref{l:welld}] 
By contradiction, assume $\ga$ is the smallest positive cycle
with irreducible decompositions $\rho_1,\ldots,\rho_k$ and
$\rho'_1,\ldots,\rho'_\ell$, giving different multiplicities.

We claim that $\rho_1'$ must appear on the first list
as $\rho_i$ and does not intersect (edge-wise) any of the
previous cycles $\rho_j$, $j<i$.  Indeed, neither $\rho_j$
can contain $\rho_1'$ or vice versa since both are irreducible.
However, if they have non-empty overlap, one of them must
contain the end of another which contradicts positivity.
Since the edges of $\rho_1'$ have to be eventually removed,
we have the claim.

By construction, we now have a new positive cycle $\ga'=\ga-\rho_1'$
with irreducible decompositions $\rho_1,\ldots,\rho_{i-1},\rho_{i+1},\ldots,\rho_k$
and $\rho'_2,\ldots,\rho'_\ell$, giving different multiplicities.
This contradicts the assumption that~$\ga$ is minimal.
\end{proof}

\subsection{Counting cycles}
Let $T$ be a tall set of tiles and $f(n)=\CT(\Rven)$.
Consider graph~$\scg_T$ constructed in Lemma~\ref{lemGraph}, and
let $\rho_1,\ldots,\rho_r$ be the list of irreducible cycles in~$\scg_T$,
ordered lexicographically.
Denote by $\BT(z_1,\ldots,z_r)$ the number of cycles~$\ga$ in $\scg_T$,
which start at~$v_0$ and have multiplicity $\mult(\rho_i,\ga)=z_i$.

For each $0<j<i$, let $a_{i,j}$ be the number of times the
first vertex in $\rho_i$ is visited in $\rho_j$, where
the first and last vertex in $\rho_j$ is considered to be visited exactly once.
Let $a_{i,0}=1$ if the first vertex of $\rho_i$ is $v_0$
and let $a_{i,0}=0$ otherwise.

\begin{lemma}\label{l:enum-BT}
We have
$$
\BT(z_1,\ldots,z_r)\,=\,
\prod_{i=1}^{r}{a_{i,0}+a_{i,1}z_1+\ldots+a_{i,i-1}z_{i-1}+z_i-1\choose z_i}.
$$
\end{lemma}

\begin{proof}
Given a cycle $\gamma$ in $\scg_T$ starting at~$v_0$ with
$\mult(\rho_j,\gamma)=z_j$ for all~$j$,
we can remove irreducible subcycles of $\gamma$ one at a time,
until we are left with the empty cycle at~$v_0$.
Reversing the process, we can also speak of ``adding''
irreducible cycles to build~$\ga$.

Let~$i$, $0\le i\le r$, be maximal index, such that $z_i>0$.
By definition, every vertex in $\rho_i$ has index greater than every vertex
at the start of a irreducible cycle with positive multiplicity in~$\gamma$.
Thus, irrespectively of order in which we add the irreducible cycles,
no cycles are inserted in the middle of a copy of $\rho_i$.  Therefore,
we may assume that the copies of $\rho_i$ are added last.
Further, one can take the cycle $\gamma$ and determine the
cycle with all of the copies of $\rho_i$ removed, and the locations
where the $\rho_i(\gamma)$ copies of $\rho_i$ were inserted.

Let $\gamma^\prime$ be the cycle $\gamma$ with all copies of $\rho_i$ removed,
and let $v_k$ be the start of~$\rho_i$. Note that the number
of $v_k$ in $\gamma^\prime$ is exactly
$$
a_{i,0}+a_{i,1}z_1+\ldots+a_{i,i-1}z_{i-1}\ts,
$$
since adding the cycle $\rho_j$ adds $a_{i,j}$ more vertices~$v_k$.
Therefore, the number of ways to add $z_i$ copies of $\rho_i$
to $\gamma^\prime$ is equal to
\begin{equation}
{a_{i,0}+a_{i,1}z_1+\ldots+a_{i,i-1}z_{i-1}+z_i-1\choose z_i}\ts.
\tag{$\star$}
\end{equation}
We conclude that $\BT(z_1,\ldots,z_r)$ is ($\star$) times the number
of possible strings you can get after removing all $z_i$ copies of~$\rho_i$.
This gives the recursive formula:
$$
\BT(z_1,\ldots,z_i,0,\ldots,0)\.=\. {a_{i,0}+a_{i,1}z_1+\ldots+a_{i,i-1}z_{i-1}+z_i-
1\choose z_i}\BT(z_1,\ldots,z_{i-1},0,\ldots,0)\ts.
$$
Since $B_T(0,\ldots,0)=1$, iterating the above formula gives the result.
\end{proof}

\smallskip

We can now count all cycles~$\ga$ which start at~$v_0$ by summing over all
lists of irreducible cycles as above giving decompositions of~$\ga$.

\begin{lemma}\label{countR0}
Every tile counting function $f\in \cF$ can be written as
$$
f(n)\, =\, \sum \. \prod_{i=1}^{r}\.
{a_{i,0}+a_{i,1}z_1+\ldots+a_{i,i-1}z_{i-1}+z_i-1\choose z_i}\.,
$$
where the sum is over all $(z_1,\ldots,z_r)\in\mathbb{Z}^r$ satisfying
$c_1 z_1+\ldots+c_r z_r=n+\varepsilon$, where all $c_i\in\mathbb{R}$
and $a_{i,j}\in\mathbb{N}$.
\end{lemma}

\begin{proof}
By Lemma~\ref{lemGraph}, the function $f(n)$ counts the number of cycles~$\ga$ in $\scg_T$
which start at~$v_0$ of weight $n+\varepsilon$. In notation above, we have for such~$\ga$~:
$$
w(\ga) \, = \, w(\rho_1)\mult(\rho_1,\gamma)+\ldots+w(\rho_r)\mult(\rho_r,\gamma)
\, = \, n+\varepsilon\ts.,
$$
Therefore,
$$
f(n) \, = \, \sum \. \BT(z_1,\ldots,z_r)\.,
$$
where the summation is over all $(z_1,\ldots,z_r)$ such that
$w(\rho_1)z_1+\ldots+w(\rho_r)z_r \, = \, n+\varepsilon$.
Now Lemma~\ref{l:enum-BT} implies the result.
\end{proof}

\subsection{Proof of Lemma~\ref{countR1}}
In notation above, denote by $Z_n$ the set of of all vectors
$\bz = (z_1,\ldots,z_r)\in\mathbb{Z}^r$ satisfying
$c_1 z_1+\ldots+c_r z_r=n+\varepsilon$, where all $c_i\in\mathbb{R}$
and $a_{i,j}\in\mathbb{N}$.
By Lemma \ref{countR0}, every $f\in \cF$ can be written as
$$
f(n)\, =\, \sum_{\bz \in \Z_n} \. \prod_{i=1}^{r}\.
{a_{i,0}+a_{i,1}z_1+\ldots+a_{i,i-1}z_{i-1}+z_i-1\choose z_i}\..
$$

Without loss of generality, we may assume that $c_r=\varepsilon$ and $c_{r-1}=1$, since if this were not the case, we could add two tiles to $T$ of area $\varepsilon$ and $1$, each with a new boundary that only fits together with itself. This adds two disjoint loops to $\scg$, and we can label the vertices so that these two disjoint loops are the last two irreducible cycles. Note that for any $n$, the set $Z_n$ is nonempty. In particular, it contains the vector $(0,0,\ldots, n,1)$.

Consider the set $W\ssu \zz^r$ of all integer vectors $(w_1,\ldots w_r)$ with $c_1 w_1+\ldots+c_r w_r=0$. This set forms a lattice, and therefore has a basis, ${\bf b}_1,\ldots, {\bf b}_d$. Note that the set $Z_{n}$ is exactly the set of all vectors $\bz = v_1{\bf b}_1+v_2{\bf b}_2+\ldots+v_d{\bf b}_d+(0,\ldots,n,1)$, with each $v_i\in\zz$, and each vector is expressible uniquely in this way. Thus, each coordinate $z_i=\be_i(v_1,\ldots,v_d,n)$ is an integer coefficient affine function of $(v_1,\ldots,v_d,n)$.  This implies
$$a_{i,0}+a_{i,1}z_1+\ldots+a_{i,i-1}z_{i-1}+z_i-1 \, = \, \al_i(v_1,\ldots,v_d,n)\.,
$$
where $\al_i$ is also integer coefficient affine functions of  $(v_1,\ldots,v_d,n)$.
Therefore,
$$f(n) = \sum_{v\in\mathbb{Z}^{d}}\,\ts\prod_{i=1}^{r}\. {\alpha_{i}(v,n)\choose\beta_{i}(v,n)}, $$
where $\alpha_{i},\beta_{i}$ and $r, d\in\pp$ are as desired. \ $\sq$

\bigskip

\section{Diagonals of $\mathbb{N}$-rational functions are tile counting functions}\label{s:Diag}

\nin
In this section, we make the next step towards
the proof of Main Theorem~\ref{t:main-countR}.

\begin{lemma}\label{D0}
$\cN\subseteq\mathcal{F}$.
\end{lemma}

\smallskip

\nin
In other words, we prove that every diagonal~$f$
of an $\mathbb{N}$-rational generating function,
is also a tile counting function.

\subsection{Paths in networks} \label{Dproofs}

Let $\G=(V,E)$ be a directed weighted multi-graph with a unique
\emph{source}~$v_1$ and \emph{sink}~$v_2$. Further, assume that the
edges of~$\G$ are colored with $k$ colors.  We call such graph a
\emph{$k$-network}.  We need the following technical lemma.

\begin{lemma}\label{D6}
Let $G(x_1,\ldots,x_k)\in \mathcal{R}_k$.
Then there exists a $k$-network~$\G$, such that for all
$(n_1,\ldots,n_k)\in \nn^k$, $n_1+\ldots+n_k\ge 1$,
the number of paths from $v_1$ to $v_2$ with exactly
$n_i$ edges of color~$i$ is equal to
$\bigl[x_1^{n_1}\ldots x_k^{n_k}\bigr]\,G$.
\end{lemma}

\begin{proof} 
Let $\mathcal{Q}_k$ be the set of GFs, for which
there is a $k$-network as in the lemma.  We show that
$\mathcal{Q}_k$ satisfies the three conditions
in the definition of $\mathbb{N}$-rational generating
function, which proves the result.

Condition~$(1)$ is trivial.  To get $0\in \mathcal{R}_k^\prime$,
take the graph with vertices $v_1$ and $v_2$ and no edges.
Similarly, to get $x_i\in \mathcal{R}_k^\prime$, take the graph with
vertices $v_1$ and $v_2$ and a unique edge $(v_1,v_2)$ of color~$i$.

For~$(2)$, let $F,G \in \mathcal{Q}_k$, and let $\U,\W$ be the
corresponding $k$-networks.  
Attaching sinks and sources, and the rest of $\U$ and~$\W$ in parallel,
gives $F + G\in \mathcal{Q}_k$ (see Figure~\ref{f:nets}).
Similarly, if $[1]\ts F=[1]\ts G=0$,
attaching  $\U$ and~$\W$ sequentially gives GF $F \cdot G$.
More generally, for $a=[1]\ts F$ and $b=[1]\ts G$ not necessarily zero,
write $a\ts G = G +\ldots +G$ ($a$ times), and use
$$
F\cdot G \. = \. (F-a) \cdot (G-b) \. + \. a\ts G \. + \. b\ts F\ts.
$$
to obtain the desired $k$-network.

For~$(3)$, let $F \in \cQ_k$ and $[1]\ts F=0$.  To obtain $1/(1-F)$, write:
$$
\frac{1}{1-F} \, = \, 1+ \.  F \. + \. F\cdot F \. + \. \frac{F^3}{1-F^2}\. + \. F\ts \cdot\. \frac{F^3}{1-F^2}\,.
$$
For $\,\frac{F^3}{1-F^2}$\., arrange four copies
of $k$-network~$\U$ as shown in Figure~\ref{f:nets}.
The details are straightforward.
\end{proof}

\begin{figure}[hbt]
\begin{center}
%
%
\psfrag{F}{\small $\U$}
\psfrag{G}{\small $\W$}
\psfrag{F1}{\small $F\cdot G$}
\psfrag{F2}{\small $F + G$}
\psfrag{F0}{\small $F$}
\psfrag{FG}{\small $G$}
\psfrag{F3}{\small $\frac{F^3}{1-F^2}$}
\epsfig{file=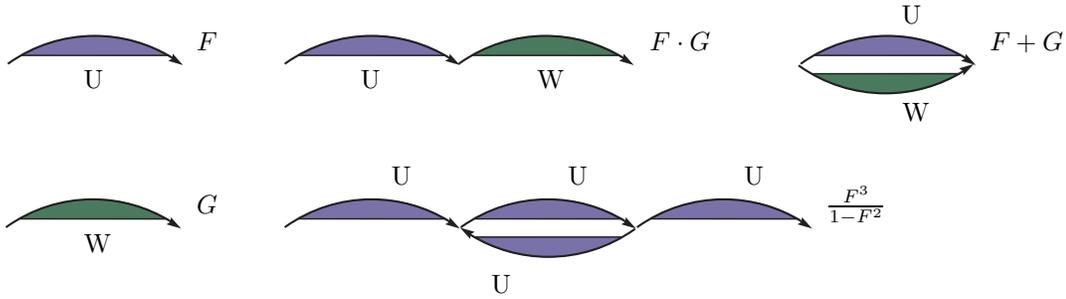,width=13.6cm}
\end{center}
\caption{Networks giving \ts $F\cdot G$, \ts $F+G$ \ts and \ts $F^3/(1-F^2)$.}
\label{f:nets}
\end{figure}

\subsection{Proof of Lemma ~\ref{D0}}
Let $G\in \mathcal{R}_k$ be such that
$f(n)=\bigl[x_1^{n}\ldots x_k^{n}\bigr]\,G$.
By Lemma~\ref{D6}, there is a $k$-network $\W$ with source~$v_1$
and sink~$v_2$, such that there are exactly~$f(n)$ paths from $v_1$
to~$v_2$, which pass through $n$ edges of color~$i$, for all~$i$.

Let $\varepsilon,\alpha_1,\ldots,\alpha_k>0$ be irrational numbers,
such that the only rational linear dependence between them is
$\alpha_1+\ldots+\alpha_k=1$. Assign weight~$\al_i$ to each edge
in~$\W$ with color~$i$.  Add vertex~$v_0$ and edges $(v_0,v_1)$
$(v_2,v_0)$, both of weight $\varepsilon/2$.  Denote the resulting
graph by $\hW$.

Note that cycles in~$\hW$ which start at~$v_0$ of \tweight $n+\varepsilon$
are in bijection with paths from $v_1$ to $v_2$ in~$\W$ with exactly~$n$
edges of each color. Therefore, there are exactly $f(n)$ of them.

In notation of the proof of Lemma~\ref{lemGraph}, we associate a
different tile boundary $\partial_i$ with vertices $v_i$, $1\le i \le N$,
and the vertical line segment with the vertex~$v_0$.  We can always
ensure \ts width$(\partial_i) < \frac1{3}\min\{\ve/2,\al_1,\ldots,\al_k\}$.

For every edge $e=(v_i,v_j)$ in~$\hW$ with weight $w_e$, denote by $\tau_e$
the unique tile with height~1, $\partial_L(\tau_e)=\partial_i$,
$\partial_R(\tau_e)=\partial_j$ and area $|\tau_e|=w_e$.  Note that such
tile exists by the width condition above.  Let $T$ be the set of
tiles~$\tau_e$.  From above, for all $n\geq 1$, the number of tilings
of $\Rven$ by~$T$ is equal to the number of cycles in $\hW$ starting
at $v_0$ of \tweight $n+\varepsilon$, which is equal to $f(n)$ by assumption.

When $n=0$, this tile set has zero tilings. Since \ts $a:= [1]\ts F\in \nn$,
we can make the number of tilings of $\textrm{R}_\varepsilon$ equal to $a$
by adding $a$ copies of a $1\times\varepsilon$ rectangle to $T$.
This does not change the number of tilings for any $n\geq 0$,
since every tiling for $n\geq0$ must already has two tiles of area~$\varepsilon/2$,
and thus cannot contain more tiles of area~$\ve$.

Finally, if $T$ has multiple copies of the same tile, replace each copy
with two tiles which only fit together with each other, to make a copy
of that tile.  We can always do this in such as way to make all new
tiles distinct.  This implies that $f\in \cF$, as desired.
\ $\sq$

\bigskip

\section{Binomial multisums are diagonals of $\mathbb{N}$-rational functions}\label{s:BinDiag}

\nin
In this section, we prove the following result towards the proof of Main Theorem~\ref{t:main-countR}.

\begin{lemma}\label{D1}
$\mathcal{B}\subseteq\cN$.
\end{lemma}

The proof of the lemma follows easily from five sub-lemmas, three
on diagonals and two on binomial multisums.  While the
former are somewhat standard, the latter are rather technical;
we prove them in the next section.

\subsection{Diagonals}
We start with the following three simple results.

\begin{lemma}\label{l:D2-multi}
The set of diagonals of an $\mathbb{N}$-rational generating functions
is closed under addition and multiplication.
\end{lemma}

\begin{proof} Let $f, g \in \cN$.  We have
$$
f(n)=\bigl[x_1^{n}\ldots x_{k}^{n}\bigr]\,F(x_1,\ldots,x_{k}), \quad
g(n)=\bigl[y_1^{n}\ldots y_{\ell}^{n}\bigr]\,G(y_1,\ldots,y_{\ell})\.,
$$
for some $F\in \cR_k$ and $G\in \cR_\ell$. Consider
$A(x_1,\ldots x_{k},y_1,\ldots,y_{\ell})$ defined as
$$
A(x_1,\ldots x_{k},y_1,\ldots,y_{\ell})=
\left(\prod_{i=1}^{\ell}\frac{1}{1-y_i}\right)F\bigl(x_1,\ldots,
x_{k}\bigr)+\left(\prod_{i=1}^{k}\frac{1}{1-x_i}\right)G\bigl(y_1,\ldots,y_{\ell}\bigr)\..
$$
It follows from the definition of $\nn$-rational functions, that $A\in \mathcal{R}_{k+\ell}$.
We have
$$\aligned
& \bigl[x_1^{n}\ldots x_{k}^{n}y_1^{n}\ldots y_{\ell}^{n}\bigr]\,A \, = \,
\bigl[x_1^{n}\ldots x_{k}^{n}\bigr]\, F\bigl(x_1,\ldots,x_{k}\bigr) \\
& \quad \. + \. \bigl[y_1^{n}\ldots y_{\ell}^{n}\bigr]\, G\bigl(y_1,\ldots,y_{\ell}\bigr)
\, = \, f(n) + g(n)\..
\endaligned
$$
Similarly, define
$$B(x_1,\ldots x_{k},y_1,\ldots,y_{\ell}) \, = \, F\bigl(x_1,\ldots,
x_{k}\bigr) \cdot G\bigl(y_1,\ldots,y_{\ell}\bigr)\ts,
$$
and observe that
$$
\bigl[x_1^{n}\ldots x_{k}^{n}y_1^{n}\ldots y_{\ell}^{n}\bigr]\,B \,
= \, f(n) \cdot g(n)\.,
$$
as desired. \end{proof}

A function $f$ is called a {\em quasi-diagonal} of an $\mathbb{N}$-rational generating function if
$$
f(n)\. = \. \bigl[x_1^{c \ts n}\ldots x_{k}^{c\ts n}\bigr]\,F(x_1,\ldots,x_{k})\ts,
$$
for some fixed constant $c\in\mathbb{P}$.

\begin{lemma}\label{D7}
For every function $f$ which is the quasi-diagonal of $F\in \cR_k$,
there exists $\ell\in \pp$ and $G\in \cR_\ell$, such that $f$ is
the diagonal of~$G$.
\end{lemma}
\begin{proof}
First, we show that for every~$F\in \cR_k$,
there exists a function $\Fpr \in \cR_{k+1}$,
such that for all $c_0, c_1,\ldots, c_k\in\mathbb{N}$~:
\[\bigl[x_0^{c_0} x_1^{c_1} x_2^{c_2}\ldots x_{k}^{c_k}\bigr]\,\Fpr(x_0,x_1,x_2,\ldots,x_{k})=
  \begin{cases}
  \bigl[x_1^{c_0+c_1}x_2^{c_2}\ldots x_{k}^{c_k}\bigr]\,F(x_1,\ldots,x_{k}) & \text{if } c_1\geq 1\ts,\\
   {\ \. 0  }     & \text{otherwise\ts.}
  \end{cases}
\]
We prove this by structural induction in the definition of~$\cR_k$.
For case~$(1)$, let $\Fpr=0$ for $F=x_i$, $i\ge 2$, and $\Fpr = x_1$
for $F = x_1$.  In case~$(2)$, let $\Fpr=\Gpr+\Hpr$ for $F=G+H$.
For~$F=G\cdot H$, let
$$
\Fpr(x_0,\ldots,x_{k})\. =\. \Gpr(x_0,\ldots,x_{k})\ts H(x_1,x_2,\ldots,x_k)\ts + \ts
G(x_0, x_2,\ldots,x_{k})\ts\Hpr(x_0,\ldots,x_k)\ts.
$$
For case~$(3)$, for $F=1/(1-G)$, let
$$
\Fpr(x_0,\ldots,x_{k})\. = \. F(x_0,x_2,\ldots,x_k)\ts\Gpr(x_0,\ldots,x_{k})\ts F(x_1,x_2,\ldots,x_k).
$$

\nin
It is clear that this works for case~(1), for $F=G+H$, and when $c_1=0$, so we only need to worry about the cases where $c_1>0$ and where $F=G\cdot H$ or $F=1/(1-G)$.

For $F=G\cdot H$, recall that
$$\bigl[x_1^{c_0+c_1}x_2^{c_2}\ldots x_{k}^{c_k}\bigr]\,F(x_1,\ldots,x_{k})$$ equals the sum over all $(d_1+e_1,\ldots, d_k+e_k)=(c_0+c_1,c_2,\ldots, c_k)$
$$\bigl[x_1^{d_1}x_2^{d_2}\ldots x_{k}^{d_k}\bigr]\,G(x_1,\ldots,x_{k})\bigl[x_1^{e_1}x_2^{e_2}\ldots x_{k}^{e_k}\bigr]\,H(x_1,\ldots,x_{k}).$$
Note that in the formula for $\Fpr,$ all instances of $x_0$ in $\Gpr(x_0,\ldots,x_{k})\ts H(x_1,x_2,\ldots,x_k),$ come from the $\Gpr(x_0,\ldots,x_{k})$. Thus, if there are $d_1$ instances of $x_0$ or~$x_1$ in the first summand, exactly $d_1-c_0$ instances of $x_1$ come from $\Gpr$. Similarly, if there are $e_1$ instances of $x_0$ or~$x_1$ in the second summand, exactly $e_1-c_1$ instances of $x_0$ must come from $\Hpr$.

We break the contributions to $\Fpr$ into two cases: $(a)$ with $d_1>c_0$, and $(b)$ with $d_1\leq c_0$. In the case~$(a)$, note that
$$\bigl[x_0^{c_0}x_1^{d_1-c_0}x_2^{d_2}\ldots x_{k}^{d_k}\bigr]\,\Gpr(x_0,\ldots,x_{k})\. =\. \bigl[x_1^{d_1}x_2^{d_2}\ldots x_{k}^{d_k}\bigr]\,G(x_1,\ldots,x_{k})\ts,
$$ 
and
$$\bigl[x_0^{e_1-c_1}x_1^{c_1}x_2^{e_2}\ldots x_{k}^{e_k}\bigr]\,\Hpr(x_0,\ldots,x_{k})\. =\. 0\ts.
$$ 
In the case~$(b)$, we similarly have:
$$\bigl[x_0^{c_0}x_1^{d_1-c_0}x_2^{d_2}\ldots x_{k}^{d_k}\bigr]\,\Gpr(x_0,\ldots,x_{k})\.=\. 0\ts,
$$ 
and
$$
\bigl[x_0^{e_1-c_1}x_1^{c_1}x_2^{e_2}\ldots x_{k}^{e_k}\bigr]\,\Hpr(x_0,\ldots,x_{k})\. 
=\. \bigl[x_1^{e_1}x_2^{e_2}\ldots x_{k}^{e_k}\bigr]\,H(x_1,\ldots,x_{k})\ts.
$$
Therefore, in the case~$(a)$, we have
$$
\bigl[x_0^{c_0}x_1^{d_1-c_0}x_2^{d_2}\ldots x_{k}^{d_k}\bigr]\,\Gpr(x_0,\ldots,x_{k})\bigl[x_1^{e_1}x_2^{e_2}\ldots x_{k}^{e_k}\bigr]\,H(x_1,\ldots,x_{k})
$$
$$
=\. \bigl[x_1^{d_1}x_2^{d_2}\ldots x_{k}^{d_k}\bigr]\,G(x_1,\ldots,x_{k})\bigl[x_1^{e_1}x_2^{e_2}\ldots x_{k}^{e_k}\bigr]\,H(x_1,\ldots,x_{k})\ts,
$$ 
and
$$
\bigl[x_1^{d_1}x_2^{d_2}\ldots x_{k}^{d_k}\bigr]\,G(x_1,\ldots,x_{k})
\bigl[x_0^{e_1-c_1}x_1^{c_1}x_2^{e_2}\ldots x_{k}^{e_k}\bigr]\,\Hpr(x_0,\ldots,x_{k})\.=\. 0\ts.
$$ 
In the case~$(b)$, we get a similar result with the r.h.s.'s interchanged.  
We conclude:  
$$
\bigl[x_0^{c_0} x_1^{c_1} x_2^{c_2}\ldots x_{k}^{c_k}\bigr]\,
\Fpr(x_0,x_1,x_2,\ldots,x_{k})\.=\.
\bigl[x_1^{c_0+c_1}x_2^{c_2}\ldots x_{k}^{c_k}\bigr]\,F(x_1,\ldots,x_{k})\ts.
$$

For case~(3), we think of any contribution to $F$ as coming from $G^r$ for some $r\in \nn$, 
and break into cases based on how many copies of $G$ we go through in this $G^r$ before 
we have seen more than $c_0$ instances of $x_1$.  We then proceed similarly.

\smallskip

Now, for every fixed $m\ge 2$, $F\in \cR_k$ and a function~$f(n)$ defined as
$$
f(n)\.=\.\bigl[x_1^{m \ts n}\ldots \ts x_{k}^{m\ts n}\bigr]\,F(x_1,\ldots,x_{k})\ts,
$$
we may recursively apply the above result to split each variable $x_i$ into~$m$
variables $x_{i\ts j}$, for $j=1,\ldots,m$.  We get a function $G\in\cR_{m\ts k}$
satisfying
$$
\bigl[x_{1 1}^{c_{11}}\ts x_{12}^{c_{12}} \cdots \ts x_{k\ts m}^{c_{k\ts m}}\bigr]\,G(x_{11},x_{12}, \ldots, x_{k\ts m})
\. = \. \bigl[x_1^{c_1}\ldots x_{k}^{c_k}\bigr]\,F(x_1,\ldots,x_{k}),
$$
whenever $c_{i1}+\ldots+c_{i\ts m}=c_i$ and $c_{i\ts j}\geq 1$, for all $i$ and~$j$.
In particular, for all $c_{i\ts j}=n\ge 1$, this gives
$$
\bigl[x_{11}^n \ts x_{12}^n \ts \cdots \ts x_{k\ts m}^n\bigr]\,G(x_{11},x_{12},\ldots,x_{k\ts m})
\. = \. f(n)\ts.
$$
Further $[1]\ts G =0$, so we can simply add
the constant term $f(0)$ to get the desired function with the diagonal~$f(n)$.
\end{proof}

\begin{lemma}\label{D8}
Let $f\in \cN$ and $g: \nn \to \nn$ satisfy $f(n)=g(n)$ for all
$n\ge 1$.  Then $g \in \cN$.
\end{lemma}

\begin{proof}
The functions
\[
 j(n)=
  \begin{cases}
  1 & \text{if } n= 0,\\
   0       & \text{otherwise}
  \end{cases}
  \qquad \text{and} \qquad
 h(n)=
  \begin{cases}
  0 & \text{if } n = 0,\\
  1        & \text{otherwise}
  \end{cases}
\]
are trivially diagonals of functions $1$ and \ts $x/(1-x)\in \cR_1$.  Writing
$$
g(n) \. = \. g(0)\ts j(n) \cdot f(n) \. + \.  h(n) \cdot f(n)\.,
$$
implies the result by Lemma~\ref{l:D2-multi}.

\subsection{Finiteness of binomial multisums}
Next, we find a bound on which terms can contribute to a binomial multisum.

\begin{lemma}\label{D4}
In notation of~\,\.$\S\ref{ss:three-binom}$, let
\begin{displaymath}
f(n)\,=\,\sum_{v\in\mathbb{Z}^{d}}\,\prod_{i=1}^r{\alpha_i(v,n)\choose\beta_i(v,n)}
\end{displaymath}
be finite for all $n\in\mathbb{N}$.
Then there exists a constant $c\in\mathbb{N}$, such that for all $n\in\pp$,
and $|v_i|>c\ts n$ for all~$i$, the product on the right hand side is zero.
\end{lemma}

Finally, we show that binomial sums bounded as in Lemma~\ref{D4}
are in fact quasi-diagonals of $\mathbb{N}$-rational generating functions.

\begin{lemma}\label{D5}
In notation of~\,\.$\S\ref{ss:three-binom}$, let
$f(n):\mathbb{N}\rightarrow\mathbb{N}$ be defined as
\begin{displaymath}
f(n)\, = \, \sum_{v\in\mathbb{Z}^{d},|v_i|\leq c\ts n} \,
\prod_{i=1}^r{\alpha_i(v,n)\choose\beta_i(v,n)}\ts,
\end{displaymath}
for all $n\in \pp$.
Then $f(n)$ agrees with a quasi-diagonal of an $\mathbb{N}$-rational generating function
at all $n\geq 1$.
\end{lemma}

Both lemmas are proved in the next section.

\subsection{Proof of Lemma~\ref{D1}}
By Lemmas~\ref{D4} and~\ref{D5}, every function $f(n)$ as in Lemma~\ref{D4}
agrees with a quasi-diagonal of an $\mathbb{N}$-rational generating function
at all $n\geq 1$. By Lemma~\ref{D7}, any function of this form agrees with
a diagonal of an $\mathbb{N}$-rational generating function
at all $n\geq 1$. By Lemma~\ref{D8}, any such function is in $\cN$.
\end{proof}

\bigskip

\section{Proofs of lemmas~\ref{D4} and~\ref{D5}}\label{s:tech-proofs}

\subsection{A geometric lemma}
We first need the following simple result; we include a short proof
for completeness.

\begin{lemma}\label{bnd}
Let $\ts\alpha_1,\ldots,\alpha_r: \mathbb{R}^{d} \to \mathbb{R}\ts$
be integer coefficient affine functions.
Let $P\ssu \rr^d$ be the (possibly unbounded)
polyhedron of points satisfying $\alpha_i\geq0$ for all~$i$.
If~$P$ contains a positive finite number of integer lattice points,
then $P$ is bounded.
\end{lemma}

\begin{proof}  Suppose~$P$ is not bounded.  Without loss of generality,
assume that at the origin $O \in P$.  Consider the \emph{base cone}
$C_P$ of all infinite rays in $P$ starting at~$O$
(see e.g.~\cite[$\S$25.5]{Pak}). Since $P$ is a rational polyhedron,
the cone~$C_P$ is also rational and contains at least one ray of
rational slope. This ray contains an integer point and has a
rational slope, and therefore contains infinitely many integer
points, a contradiction. \end{proof}

\subsection{Proof of Lemma~\ref{D4}}
Let $S$ be a subset of $\{1,\ldots,r\}$, and let $P_S$ be the set of all points
${(v,n)\in\mathbb{R}^{d+1}}$ satisfying $\alpha_i(v,n),\beta_i(v,n)\geq0$ for
$i\not\in S$, satisfying $\alpha_i(v,n)=-1$ and $\beta_i(v,n)=0$ for $i\in S$,
and satisfying $n\geq 0$. Let $|v|$ denote $\max_i |v_i|$.

Note that we have $2^r$ polytopes $P_S$, and the integer lattice points in $P_S$
form a cover for the set of all $(v,n)\in\mathbb{Z}^{d+1}$ which contribute
a positive amount to~$f(n)$. For each $P_S$, we prove that there exists
a constant $c$ such that any integer lattice point in $P_S$, with $n\geq 1$
satisfies $|v_i|\leq cn$. Since $f$ is finite, we know that $P_S$ contains
finitely many integer lattice points for any fixed value of $n$.

We can assume that there are two distinct values $n_1<n_2$, such that there
exist integer lattice points in $P_S$ with $n=n_1$ and with $n=n_2$, since
otherwise $P_S$ only contains finitely many lattice points.
Let $(v_1,n_1)$ be an integer point in $P_S$. Consider the set of all points
in~$P_S$ satisfying $n=n_2$. This is a not necessarily bounded polytope with
a positive finite number of integer lattice points, so by Lemma~\ref{bnd},
it is bounded. Thus, there exists a $c_S$, such that $|v_2-v_1|<c_S$
for all $(v_2,n_2)$ in~$P_S$.

This implies that $|v-v_1|<c_S(n-n_1)$ for all $(v,n)$ in $P_S$ with $n>n_2$.
Indeed, otherwise the line segment connecting $(v_1,n_1)$ to $(v,n)$
would intersect the hyperplane $n=n_2$ at a point $(v_2,n_2)$ in~$P_S$
but not satisfying $|v_2-v_1|<c_S$.

Take a $c^\prime_S$ such that $c^\prime_S>c$, and all finitely many
integer lattice points $(v,n)$ in $P_S$ with $1\leq n<n_2$, including
$(v_1,n_1)$, satisfy $|v|\leq c^\prime_Sn$. Then, all integer lattice
points $(v,n)$ in $P_S$ with $n\geq 1,$ satisfy $|v|\leq c^\prime_Sn$.
Taking $\ts c = \max_S \{c^\prime_S\}$, proves the result.
\ $\sq$

\subsection{Proof of Lemma~\ref{D5}}
Let $f(n):\mathbb{N}\rightarrow\mathbb{N}$ be a function such that
for all $n\geq 1$,
\begin{displaymath}f(n)= \sum_{v\in\mathbb{Z}^{d},|v_i|\leq cn}
\prod_{i=1}^r{\alpha_i(v,n)\choose\beta_i(v,n)}\,,
\end{displaymath}
where $\alpha_i$ and $\beta_i$ are integer coefficient affine
functions of $v$ and~$n$.  We construct a function $g$ in $d+2r+1$ variables
$x_1,\ldots,x_d,a_1,\ldots a_r, b_1,\ldots,b_r,y$.
Let
$$
G(x_1,\ldots,x_d,a_1,\ldots a_r, b_1,\ldots,b_r, y)\, = \, \Pi_1\cdot \Pi_2\cdot \Pi_3\cdot\Pi_4\.,
$$
where
$$
\Pi_1=\prod_{j=1}^d\. \frac{1}{1-x_jh_j}\,, \qquad \Pi_2=\prod_{j=1}^d\. \frac{1}{1-x_jh_j^\prime}\,,
$$
$$
\Pi_3=\prod_{i=1}^r \left(1+a_i\.\frac{a_i+b_ia_i}{1-(a_i+b_ia_i)}\right)\,,
\quad \text{and} \quad \Pi_4=yq\.\frac{1}{1-yq^\prime}\,.
$$
The terms $h_j$, $h_j^\prime$, $q$, and $q^\prime$ are monomials
in variables~$a_i$ and $b_i$, to be determined later.

We consider the coefficients of terms in which the exponent on
each $x_j$ variable is $2cn$. The $\Pi_1$ part will contribute some
number of these $x_j$ factors, and the $\Pi_2$
will contribute the rest. This choice will represent the variable $v_j$.
We will use $cn+v_j$ to denote the number of factors of $x_j$ coming
from~$\Pi_1$, so then $cn-v_j$ will be the number of factors of $x_j$
coming from $\Pi_2$. Note that $v_j$ can be any integer between $-cn$
and~$cn$.

Define all the monomials in such a way that the $a_i$ monomial
needs to be repeated $\alpha_i(v,n)+1$ times in the $\Pi_3$ term, while
that $b_i$ monomial needs to be repeated $\beta_i(v,n)$ times in the $\Pi_3$ term.
By the definition in~$\S$\ref{ss:def-binom}, this is exactly $\alpha_i\choose\beta_i$.

We choose the monomials $h_j$, $h_j^\prime$, $q$, and $q^\prime$ a follows.
Let
$$
\beta_i(v,n)=\beta_{i,0}+\beta_{i,1}v_1+\ldots+\beta_{i,d}v_d+\beta_{i,d+1}n.
$$
First, consider the case where $\beta_{i,j}\leq 0$, for all $0\leq j\leq d$.
In this case, we put $b_i$ in $h_j$ with multiplicity $|\beta_{i,j}|$,
put $\beta_i$ in $q$ with multiplicity $|\beta_{0,j}|$, and not put any
$b_i$ terms in $h_j^\prime$ or $q^\prime$. This implies that outside
of the $\Pi_3$ term, the number of times $b_i$ appears is exactly
$$
-\beta_{i,0}\ts + \ts \sum_{j=1}^d\ts (cn+v_j)(-\beta_{i,j}) \. = \.
n\left(-\beta_{i,d+1} \ts -\ts \sum_{j=1}^dc\ts \beta_{i,j}\right)-\ts\beta(v,n)\ts.
$$
Therefore, for the coefficient of a term with total multiplicity \.
$n\left(-\beta_{i,d+1}-\sum_{j=1}^dc\beta_{i,j}\right)$ \.
of~$b_i$, we have $\beta(v,n)$ of the $b_i$ terms must
come from the $\Pi_3$ term.

We only consider coefficients where the multiplicity of the $y$ term is~$n$.
If $\beta_{i,0}$ is positive, then we can swap the multiplicity of $b_i$ in $q$
and $q^\prime$. Since the $q^\prime$ term is necessarily repeated $n-1$ times
in the $\Pi_4$ term, this implies that outside of the $\Pi_3$ term, $b_i$
appears exactly
$$
(n-1)\beta_{i,0}\ts +\ts \sum_{j=1}^d(cn+v_j)(-\beta_{i,j}) \. = \.
n\left(\beta_{i,0}-\beta_{i,d+1}\ts -\ts \sum_{j=1}^dc\beta_{i,j}\right)-\ts \beta(v,n)
$$
times. Therefore, for the coefficient of a term with
$n\left(\beta_{i,0}-\beta_{i,d+1}-\sum_{j=1}^dc\beta_{i,j}\right)$
total multiplicity of $b_i$, we again have $\beta(v,n)$ of the $b_i$ terms must
come from the $\Pi_3$ term.

If any of the $\beta_{i,j}$ terms, with $1\leq j\leq d$ are actually positive,
we swap the multiplicity of $b_i$ in $h_j$ and $h_j^\prime$, which gives the
same analysis with $v_j$ negated. Using the same method, we can require that
the number of $a_i$ terms coming from the $\Pi_3$ term is $\alpha_i+1$,
for all~$i$.

In summary,
$$ f(n) \, = \,
\bigl[x_1^{nc_{1}}\ldots x_d^{nc_{d}}a_1^{nc_{d+1}}\ldots a_r^{nc_{d+r}}
b_1^{nc_{d+r+1}}\ldots b_r^{nc_{d+2r}} y^{nc_{d+2r+1}}\bigr]\.G\ts.
$$
Take $c^\prime$ to be a common multiple of all~$c_i$ and make a
substitution $x_i \gets x_i^{c'/c_i}$, $a_j\gets a_j^{c'/c_{d+j}}$, etc.
This gives a desired quasi-diagonal.
\ $\sq$

\bigskip

\section{Proof of Theorem~\ref{t:mainR}}\label{s:alt}

\subsection{Preliminaries}
We start with the following simple result:

\begin{lemma}\label{boundR}
Let $f\in \cF$ be a tile counting function. Then $f(n)\le C^{n}$, for all $n\in\mathbb{P}$
for some $C>0$.
\end{lemma}

\begin{proof}
Let $T=\{\tau_1,\ldots,\tau_s\}$, and let $\mu=\min_i |\tau_i|$
be the minimum area of a tile in~$T$. Every tiling of $\Rven$ with~$T$
corresponds to a unique sequence of tiles in the tilings,
listed from left to right.
The length of this sequence is at most $(n+\varepsilon)/m$.
Therefore, $f_T(n) \le (s+1)^{(n+\varepsilon)/\mu} = e^{O(n)}$.
\end{proof}

Theorem~\ref{t:mainR} shows that this upper bound is usually tight,
and every function growing slower than this must be eventually quasi-polynomial.
The following lemma is a special case of Theorem~$1.1$ in~\cite{CLS};

\begin{lemma}[\cite{CLS}]\label{l:ehr}
Let $g(n)$ be the number of integer points $(x_1,\ldots,x_{r},n)\in\mathbb{Z}^{r+1}$
satisfying $m$ inequalities $a_i(x,n)>c_i$ where $a_i$ is
an integer coefficient linear function and $c_i$ is an integer
for all $i=1,\ldots,m$.  Then $g(n)$ is eventually quasi-polynomial.
\end{lemma}

The proof of the lemma uses a generalization of \emph{Ehrhart polynomials}.
We refer to~\cite{Bar} for a review of the area and further references.

\medskip

\subsection{Proof setup}
From Main Theorem~\ref{t:main-countR}, function~$f$
can be expressed as \begin{displaymath}
(\circledast) \qquad
f(n)\, = \, \sum_{v\in\mathbb{Z}^{d}}\, \prod_{i=1}^r \. {\alpha_i(v,n)\choose\beta_i(v,n)}\ts,
\end{displaymath}
where each $\alpha_i$ and $\beta_i$ is an integer coefficient
affine function of $v$ and~$n$.

From Lemma~\ref{boundR}, function $f \le e^{cn}$ for some~$c$.
Therefore, it suffice to show that $f$ is either
greater than $e^{cn}$ for some $c>0$, or is eventually polynomial.
Furthermore, it suffices to show that $f$ is either greater
than $e^{cn}$ for some~$c$, or eventually quasi-polynomial.
We can decompose $f$ into even more functions, by multiplying $p$
by the periods of all of the quasi-polynomials, which proves that
each component function that does not grow exponentially
is eventually polynomial.

Denote by $k$ the number of indices $i$ in~$(\circledast)$, such
that $\alpha_i$, $\beta_i$, and $\gamma_i=\alpha_i-\beta_i$ are
three non-constant functions.  We use induction on~$k$.

\medskip

\subsection{Step of Induction}
Let $M$ be a constant integer satisfying $\ts M>|\beta_i(0,0)|,|\gamma_i(0,0)|$,
for all $i$. We decompose $f$ into a sum of $(M+2)^{2r}$
functions depending on the values of $\beta_i$ and $\gamma_i$, for all~$i$.
For each $\beta_i$ and $\gamma_i$, we either require that
$\beta_i\geq M$ or that the value of $\beta_i$ be some constant~$<M$.
There are $M+2$ possibilities for each function, since only values
$\ge -1$ give non-zero binomial coefficients, giving $(M+2)^{2r}$
bound as above.

To ensure that $\beta_i=z\geq-1$ for some constant~$z$,
we replace $\beta_i(v,n)$ by $z$, and multiply the binomial coefficients
${\beta_i(v,n)+1\choose z+1}$ and ${z+1\choose \beta_i(v,n)+1}$ to the
existing product. This works because
${\beta_i(v,n)+1\choose z+1}{z+1\choose \beta_i(v,n)+1}$ is~1
if $\gamma_i=z$, and 0 otherwise.

Similarly, to ensure that $\gamma_i=z\geq-1$,
we replace $\alpha_i(v,n)$ with $\beta_i+z$, and multiply
the binomial coefficients ${\gamma_i(v,n)+1\choose z+1}$ and
${z+1\choose \gamma_i(v,n)+1}$ to the existing product.

Finally, to ensure that $\beta_i\geq M$, we multiply
the binomial coefficient ${\beta_i-M-1\choose 0}$ to the existing product.
This binomial coefficient is 1 if $\beta_i-M-1\geq-1$, and 0 otherwise.
Similarly for enforcing that $\gamma_i\geq M$.

Note that when we require that $\beta_i$ or $\gamma_i$ equal
to a constant, we reduce $k$ by 1, and when we specify that
$\beta_i\geq M$ or $\gamma_i\geq M$, we keep $k$ the same.
Therefore of these $(M+2)^{2r}$ functions which add to $g$,
we get by induction that all but one of them is quasi-polynomial.
The only one we have to worry about is the function $g$
in which each $\beta_i$ and $\gamma_i$ is specified
to be at least $\mathcal{M}$, and we show that this function is identically~0.

Assume by way of contradiction that $g$ is not identically~0.
Then, there exists some $(v_1,n_1)$ such that
\begin{displaymath}
\prod_{i=1}^r{\alpha_i(v_1,n_1)\choose\beta_i(v_1,n_1)}>0, \ \
\beta_i(v_1,n_1)-\beta_i(0,0)>0, \ \ \text{and} \ \ \gamma_i(v_1,n_1)-\gamma_i(0,0) >0,
\end{displaymath}
for all~$i$.  Adding $(v_1,n_1)$ to any point $(v,n)$ would
increase every $\beta_i(v,n)$ and $\gamma_i(v,n)$ by at least~1.
Consider the sequence of points $(v_t,n_t)=(tv_1,tn_1)$,
where~$t$ is a positive integer. Note that every $\beta_i(v_t,n_t)$
and every $\gamma_i(v_t,n_t)$ is at least $t$. Therefore,
${\alpha_1(v_t,n_t)\choose\beta_1(v_t,n_t)}\geq{2t\choose t}$,
so $f(n_1t)\geq {2t\choose t}\geq 2^t$, contradicting the fact
that $f(n)<e^{cn}$ for all positive~$c$.

\medskip

\subsection{Base of Induction:}
Now consider the case $k=0$.  In notation of~$(\circledast)$, this means
\begin{displaymath}
f(n)\, =\, \sum_{v\in\mathbb{Z}^{d}}\prod_{i=1}^r\. {\alpha_i(v,n)\choose\beta_i(v,n)}\ts,
\end{displaymath}
where for each~$i$, at least one of $\alpha_i$, $\beta_i$, and $\gamma_i$
is constant. Without loss of generality, we can assume that either
$\alpha_i$ or $\beta_i$ are constant.

When $\alpha_i$ is a nonzero constant, then we can write $f$
as a sum of functions where we condition on the value of $\beta_i$
to be~$z$, by replacing $\beta_i$ with $z$ and multiplying
the existing product by ${0\choose \beta_i(v,n)-z}$.
Since $\alpha_i\choose z$ is a constant, we can again express
our functions as a sum of $\alpha_i\choose z$ copies of that
function with the $\alpha_i\choose z$ term removed.
Therefore, we can assume that if $\alpha_i$ is constant,
that constant is~0.

We can also replace every $0\choose\beta_i(v,n)$
with ${\beta_i(v,n)-1\choose 0}{-\beta_i(v,n)-1\choose 0}$,
since we are replacing the indicator that
$\beta_i=0$ with the indicators $\beta_i-1\geq-1$ and
$-\beta_i-1\geq -1$. Therefore, we may assume that every
$\beta_i$ is a constant.

In summary,
\begin{displaymath}
f(n)\, = \, \sum_{v\in\mathbb{Z}^{d}}\, \prod_{i=1}^{r_1}\. {\alpha_i(v,n)\choose 0}\.
\prod_{i=r_1+1}^{r}\. {\alpha_i(v,n)\choose z_i}\ts,
\end{displaymath}
where each $\alpha_i$ and $\beta_i$ is an integer coefficient
affine function of $v$ and~$n$, and each $z_i\in \pp$.

Note that each ${\alpha_i(v,n)\choose 0}$ term is just an indicator
function that $\alpha_i(v,n)\geq -1$. Therefore,
\begin{displaymath}
f(n)\,=\sum_{v\in P_n}\prod_{i=r_1+1}^{r_2}{\alpha_i(v,n)\choose z_i},
\end{displaymath}
where $P_n$ is the polytope of all integer points such that
$\alpha_i(v,n)\geq -1$ for all $1\leq i\leq d_1$.

Also, note that ${\alpha_i(v,n)\choose z_i}$ is equal to the
number of integer points $(x_1,\ldots,x_{z_i})$, such that
$$
{0\leq x_1<x_2<\ldots<x_{z_i}<\alpha_i(v,n).}
$$
Therefore, $f(n)$ is equal to the number of points in the polytope
$$
P_n\times\mathbb{Z}^{z_{r_1+1}}\times\ldots\times\mathbb{Z}^{z_r}\ts,
$$
where $v$ is the point in $P_n$, and the coordinates
$(x_1,\ldots,x_{z_i})$ satisfying
$$
0\leq x_1<x_2<\ldots<x_{z_i}<\alpha_i(v,n)\ts.
$$
By Lemma~\ref{l:ehr}, $f(n)$ is eventually quasi-polynomial.
This proves the base of induction, and completes the proof
of Theorem~\ref{t:mainR}.

\bigskip

\section{Proofs of applications}\label{s:tech-apps}

\subsection{Proof of Theorem~\ref{t:main-cbb}} \label{ss:proof-balanced}
First, we show that $\mathcal{B}^\prime\subseteq \mathcal{B}$. Since $\mathcal{B}$ 
is closed under addition, it suffices to show that every balanced multisum 
$$
g(n)\,=\,\sum_{v\in\mathbb{Z}^d}\.\prod_{i=1}^r\frac{\alpha_i(v,n)!}{\beta_i(v,n)!\ts\ga_i(v,n)!}
$$ 
is in $\mathcal{B}$. This follows since
$$\frac{\alpha_i(v,n)!}{\beta_i(v,n)!\ts\ga_i(v,n)!} \,=\,
{\alpha_i(v,n)\choose\beta_i(v,n)}{\alpha_i(v,n)-1\choose 0}.
$$ 
The second factor ensures that $\alpha_i\geq 0$, so the first factor is never~$-1\choose 0$.

To show that $\mathcal{B}\subseteq \mathcal{B}^\prime$, take
$$
g(n)\,=\, \sum_{v\in\mathbb{Z}^d}\.\prod_{i=1}^r{\alpha_i(v,n)\choose\beta_i(v,n)}.
$$ 
Denote by~$S$ the set of all subsets of $\{1,\ldots,r\}$, and $\ga_i=\al_i-\be_i$. 
For each $s\in S$, let 
$$g_s(n)\, =\, \sum_{v\in\mathbb{Z}^d}\. e_s(v,n) \cdot \prod_{i\in s}\frac{\alpha_i(v,n)!}{\beta_i(v,n)!\ts\ga_i(v,n))!}\,, \qquad \text{where}
$$
$$e_s(v,n) \, = \, \prod_{i\not \in s}\,\left[\frac{0!}{(\alpha_i(v,n)+1)!(-\alpha_i(v,n)-1)!}\right]\ts 
\left[\frac{0!}{\beta_i(v,n)!(-\beta_i(v,n))!}\right].
$$
Observe that $e_s(v,n) = 1$ if $\alpha_i(v,n)=-1$ and $\beta_i(v,n)=0$, and $e_s(v,n) = 0$ otherwise.  

For every $v$, let $s_v$ be the set of indices $i\in\{1,\ldots,r\}$ for which $\alpha_i(v,n),\beta_i(v,n)\geq 0$. 
Then
$$
e_s(v,n) \cdot \prod_{i\in s}{\frac{\alpha_i(v,n)!}{\beta_i(v,n)!\ts \ga_i(v,n)!}} \, = \, {\alpha_i(v,n)\choose\beta_i(v,n)}
\quad \text{when \ $s=s_v$\.,}
$$
and $0$ otherwise.  Therefore,
$$g(n)\, = \, \sum_{s\in S}\. g_s(n)\.. 
$$ 
This implies that $g\in \mathcal{B}^\prime$, and completes the proof. \ $\sq$

\subsection{Getting close to Catalan numbers}
Before we prove Proposition~\ref{cat4}, we need the following weaker result.
\begin{lemma}\label{l:cat3}
There exists a tile counting function $f$ such that
$$
f(n) \, \sim \, \frac{3\sqrt{3}}{\pi}\. C_n\,, \ \quad \ \text{as} \quad \,\, n\to \infty\..
$$
\end{lemma}

\begin{proof} Consider the following three binomial multisums $f_1,f_2,f_3 \in \cB$~:
$$f_1(n)\, = \, \sum_{v\in\mathbb{Z}} {n\choose 3v}{3v\choose n}{2v\choose v}^3, \quad
f_2(n)\, =\, 4 \. \sum_{v\in\mathbb{Z}} {n-1\choose 3v}{3v\choose n-1}{2v\choose v}^3,
$$
$$
\text{and} \quad \ f_3(n)\, = \,16\. \sum_{v\in\mathbb{Z}} {n-2\choose 3v}{3v\choose n-2}{2v\choose v}^3.
$$
Let $f=f_1+f_2+f_3$. By Main Theorem~\ref{t:main-countR} and Corollary~\ref{c:add-mult}, we know that $f\in \cF$.

Observe that $f_1(n)\ne 0$ only when $n$ is a multiple of~3.  We have:
$$
f_1(n)\, = \, {2n/3\choose n/3}^3 \, \sim \, \left(\frac{4^{n/3}}{\sqrt{n/3}\sqrt{\pi}}\right)^3
\,\sim \, \frac{3\sqrt{3}}{\pi}\,C_n\.,\quad \mbox{ for } \ 3|n.
$$
Analyzing $f_2$ and $f_3$ gives the same result when $n=1,2$~mod~$3$, respectively.
\end{proof}

\subsection{Proof of Proposition~\ref{cat4}}
Let $f\in \cF$ be the tile counting function from Lemma~\ref{l:cat3}.
For each $i\in\mathbb{N}$, let $g_i(n)=f(n-i)$ if $n\geq i$, and let
$g_i(n)=0$ otherwise. Each $g_i$ is also a tile counting function,
since we can take the exact same tile set,
and replace $\varepsilon$ with $\varepsilon+i$.

Denote $\xi = 3\sqrt{3}/\pi$.
Note that $g_i(n)\sim f(n)/4^i\sim C_n\ts \xi/4^i$.
Given any $\varepsilon>0$, we can take $i$ large enough so that
$\xi/4^i<\varepsilon$, and $m\in \pp$ such that
$1-\ep <m\ts \xi/4^i<1+\ep$.
This gives \ts  $m\ts g_i(n) \sim C_n \ts \xi \ts m/4^i$
which is between $1-\ep$ and $1+\ep$, as desired.  Finally,
we have $m\ts g_i \in \cF$ since $\cB = \mathcal{F}$ is closed
under addition. \ $\sq$

\subsection{Proof of Proposition~\ref{cat6}}
Given an $m\geq 1$, let
$$
f(n)={2n \choose n}+(m-1){2n\choose n+1}.
$$
Note that $f$ is a tile counting function, since it is a finite sum
of binomial coefficients of affine functions of $n$. Since
$C_n={2n\choose n}-{2n\choose n+1}$, we have that $f(n)$ and $C_n$
differ by $m{2n\choose n}$, and are therefore congruent modulo~$m$.
 \ $\sq$

\subsection{Proof of Proposition~\ref{cat7}}
Given a prime $p\ge 2$, let
$$
f(n)={2n \choose n}+(p^{2n}-1){2n\choose n+1}.
$$
By Corollary~\ref{c:posapp}, $p^{2n}-1 \in \cB$, and
binomial coefficients are in $\cB$ by definition.  Since
$\cB$ is closed under addition and multiplication,
we obtain $f\in \cB$.
Note that $p^{2n}>C_n$, so adding or subtracting an integer multiple
of $p^{2n}$ to $C_n$ does not change the order of $p$. Therefore,
$$
\text{ord}_p(C_n)\. = \. \text{ord}_p\left(C_n+p^{2n}{2n\choose n+1}\right)
\. = \. \text{ord}_p(f(n))\.,
$$
as desired. \ $\sq$

\subsection{Proof of Theorem~\ref{t:hypo}}
We start with the case where $k=1$ and $\ell=2$. Let $r=r_1/r_2$,
$c=\mu_1^{\mu_1}r_2$, and let
$$
f(n)=\sum_{\ell=0}^n{\mu_1 \ell\choose \nu_1\ell}c^{n-\ell}(\nu_1^{\nu_1}\nu_2^{\nu_2}r_1)^\ell
\. = \. \sum_{\ell\in\mathbb{Z}}
{\ell-1\choose 0}{n-\ell-1\choose 0}
{\mu_1 \ell\choose \nu_1\ell}c^{n-\ell}(\nu_1^{\nu_1}\nu_2^{\nu_2}r_1)^\ell\ts.
$$
First, let us prove that $f$ is a tile counting function. Replace $c^{n-\ell}$ with
$$
\sum_{v_1,\ldots,v_{c}\in\mathbb{Z}}{n-\ell\choose v_1}{n-\ell-v_1\choose v_2}{n-\ell-v_1-v_2\choose v_3}
\ldots{n-\ell-v_1-v_2-\ldots- v_{c-1}\choose v_{c}}.
$$
We can ignore the fact that ${-1 \choose 0}=1$,
since if we take the least $i$ such that $n-\ell-v_1-v_2-\ldots- v_{i}=-1$,
we have ${n-\ell-v_1-v_2-\ldots- v_{i-1}\choose v_i}=0.$
We then make a similar replacement for $(\nu_1^{\nu_1}\nu_2^{\nu_2}r_1)^\ell$.
Therefore, $f\in\cF$ by the Main Theorem~\ref{t:main-countR}.

Letting
$$
g(\ell)={\mu_1 \ell\choose \nu_1\ell}c^{-\ell}(\nu_1^{\nu_1}\nu_2^{\nu_2}r_1)^\ell,
$$
we get
$$
f(n)\. = \. \sum_{\ell=0}^n \. g(\ell)\ts c^n\ts.
$$
Note that $g(0)=1$, and
$$
g(\ell+1)/g(\ell) \, = \,
\frac{\nu_1^{\nu_1}\nu_2^{\nu_2} \. r\. \prod_{i=1}^{\mu_1} (\mu_1\ell+i)}{\mu_1^{\mu_1}\.
\prod_{i=1}^{\nu_1} \left(\nu_1\ell+i\right) \.\ts
\prod_{i=1}^{\nu_2} \left(\nu_2\ell+i\right)} \, = \,
\frac{(\ell+a_1)(\ell+a_2)\ldots(\ell+a_p)\ts r}{(\ell+b_1)(\ell+b_2)\ldots(\ell+b_p)}.
$$
Therefore,
$$
f(n)/c^n\, = \, \sum_{\ell=0}^n\, \prod_{k=0}^{\ell-1}\frac{(k+a_1)(k+a_2)\ldots(k+a_p)\ts r}{(k+b_1)(k+b_2)\ldots(k+b_p)}
\, \to \, A \quad \text{as} \ \ \. n\to \infty\..
$$
In general, let each part $\mu_i$ be subdivided into $\ell_i$ parts
$\nu_{i,1},\ldots,\nu_{i,\ell_i}$. Let $r=r_1/r_2$, and let
$c=(\mu_1)^{\mu_1}\ldots(\mu_k)^{\mu_k}r_2$.
Similarly to in the previous case, we define $f$ as
$$
f(n)\. = \. \sum_{\ell=0}^n \. c^n \ts r_1^\ell\ts r_2^{-\ell}\.  \prod_{i=1}^p \ts g_i(\ell),
$$
where each
$$
g_i(\ell)={\mu_i\ell\choose \nu_{i,1}}{\mu_i\ell-\nu_{i,1}\ell\choose \nu_{i,2}\ell}
\ldots{\mu_i\ell-\nu_{i,1}\ell-\ldots-\nu_{i,\ell_i-1\ell}\choose\nu_{i,l_i}}
\left(\mu_i^{\mu_i}\right)^{-\ell}\left(\nu_{i,1}^{\nu_{i,1}}\ldots\nu_{i,\ell_i}^{\nu_{i,\ell_i}}\right)^\ell.
$$
This is a tile counting function for reasons similar to in the previous case.
Note that all the negative exponents are canceled out by the $c^n$ term.
We have:
$$
g_i(\ell+1)/g(\ell) \, = \,
\frac{\nu_{i,1}^{\nu_{i,1}}\ldots\nu_{i,\ell_i}^{\nu_{i,\ell_i}}
\ts\prod_{j=1}^{\mu_i} \left(\mu_i\ell+j\right)}
{\mu_i^{\mu_i}\prod_{j=1}^{\nu_{i,1}}\left(\nu_{i,1}\ell+j\right)\cdots
\prod_{j=1}^{\nu_{i,\ell_1}} \left(\nu_{i,\ell_i}\ell+j\right)}\,.
$$
Therefore,
$$
\frac{r_1^{\ell+1}\ts r_2^{-(\ell+1)}\. \prod_{i=1}^pg_i(\ell)}{r_1^\ell \ts  r_2^{-\ell} \.
\prod_{i=1}^pg_i(\ell)}=\frac{(\ell+a_1)(\ell+a_2)\. \cdots \.
(\ell+a_p)\ts r}{(\ell+b_1)(\ell+b_2)\ldots(\ell+b_p)}\,.
$$
Since
$$
r_1^0\. r_2^{0} \, \prod_{i=1}^p \. g_i(0) \, = \. 1\ts,
$$
we have
$$
f(n)/c^n=\sum_{\ell=0}^n\prod_{k=0}^{\ell-1}\frac{(k+a_1)(k+a_2)\ldots(k+a_p)\ts r}{(k+b_1)(k+b_2)\ldots(k+b_p)}
\, \to \, A \quad \text{as} \ \ \. n\to \infty\..
$$

The base of exponent we get from this construction is
$c=(\mu_1)^{\mu_1}\ldots(\mu_k)^{\mu_k}r_2$.
However, note that it is easy to multiply $c$ by any positive integer~$N$,
simply by multiplying $f$ by $N^n$. In particular, let $\pid$ be the product
of all primes which are factors of $\mu_1\ldots\mu_kr_2$. Then there exists
some positive integer~$d$, such that $\pid^d$ is a multiple of
$(\mu_1)^{\mu_1}\ldots(\mu_k)^{\mu_k}r_2$.  This implies that
there exists a function $h\in \cF$ with \ts $h(n) \sim A\ts\pid^{dn}$.

Note now, that we can scale all the tiles horizontally by $d$, and scale
$\varepsilon$ by~$d$, to get a new function $f_0(n)$ such that $f_0(dn)=h(n)$.
We may assume that $f_0(n)=0$ when $n$ is not a multiple of $d$, because we can
multiply $f_0$ by the indicator that $d|n$. We can similarly get a function $f_i$
for $i=1,\dots,d-1$, such that $f_i(n)$ is nonzero only when $n=i\mbox{ mod }d$,
and $f_i(nd+i)=\pid^i h(n)$. We have
$$
f(n)=\sum_{i=0}^{d-1}f_i(n)\in\mathcal{F},
$$
and by the way we constructed $f$ we obtain
\ts $f(n) \sim A\ts\pid^{n}$.
Thus, we can take $c=\pid$ or any integer multiple of $\pid$,
as desired. \ $\sq$

\bigskip

{\small

\section{Final Remarks}\label{s:fin}

\subsection{}
The idea of irrational tilings was first introduced by Korn,
who found a bijection between \emph{Baxter permutations}
and tilings of large rectangles with three fixed irrational
rectangles~\cite[$\S$6]{Korn}.

\subsection{}\label{ss:fin-nn}
For Theorem~\ref{t:classical}, much of the credit goes to
Sch\"{u}tzenberger~\cite{Schu} who proved the
equivalence between \emph{regular languages}
and the (weakest in power) \emph{deterministic
finite automata} (DFA).  He used the earlier work of Kleene~(1956)
and the language of semirings; the GF reformulation
in the language of $\nn$-rational functions came later,
see~\cite{SS}.  We refer to~\cite{BR,SS} for a thorough treatment of the
subject and connections to GFs, and to~\cite{Pin} for a more recent survey.

Now, the relationship between (polyomino) tilings of the strip,
regular languages (as well as DFAs) were proved more recently
in~\cite{BL,MSV}.  Theorem~\ref{t:classical} now follows as
combination of these results in several different ways.

Let us mention here that in the usual polyomino tiling setting
there is no height condition, so in fact Theorem~\ref{t:classical}
remains unchanged when rational tiles of smaller height are allowed.
In the irrational tiling setting, the standard ``finite number
of cut paths'' argument fails.  Still, we conjecture that
Theorem~\ref{t:main} also extends to tiles with smaller heights.

\subsection{}\label{ss:fin-asym}
The history of Theorem~\ref{t:growth-DFinite} is somewhat confusing.
In fact, it holds for integer \emph{$G$-sequences} defined as
integer \emph{$D$-finite $($holonomic$)$ sequences}
with at most at most exponential growth.  It is stated in this form
since diagonals of all rational GFs are $D$-finite~\cite{Ges1}
and at most exponential.  We refer to~\cite{FS,Sta} for more on
$D$-finite sequences, examples and applications, and to~\cite{DGS,Gar}
for $G$-sequences.

The asymptotics of $D$-finite GFs go back to Birkhoff and
Trjitzinsky (1932), and Turrittin (1960).
See~\cite[$\S$VIII.7]{FS} and~\cite[$\S$9.2]{Odl}
for various formulations of general asymptotic estimates,
and an extensive discussion of priority and validity issues.
However, for $G$-sequences, the result seems to be
accepted and well understood, see~\cite[$\S$2.2]{BRS}
and~\cite{Gar}.

\subsection{}\label{ss:fin-algebraic}
Note that $D$-finite sequences can be superexponential, e.g.~$n!$.
They can also have \ts $\exp(n^\ga)$ terms with $\ga \in \qqq$,
e.g.~the number $a_n$ of involutions in~$S_n$~:
$$
a_n \, \sim \,  2^{-1/2} e^{-1/4} \left(\frac{n}{e}\right)^{n/2} e^{\sqrt{n}}\ts,
$$
(see~\cite{Sta} and {\tt A000085} in~\cite{OEIS}).

In notation of Theorem~\ref{t:growth-DFinite}, the $\al \in\qqq$
conclusion cannot be substantially strengthened even for $k=2$ variables.
To understand this, recall Furstenberg's theorem (see~\cite[$\S$6.3]{Sta})
that every algebraic function is a diagonal of $P(x,y)/Q(x,y)$, and that
by Theorem~2 in~\cite{BD} there exist algebraic functions with
asymptotics $A\ts \la^n \ts n^\al$, for all $\al\in \qqq \sm \{-1,-2,\ldots\}$.
For example, the number $g(n)$ of \emph{Gessel walks} (see {\tt A135404}
in~\cite{OEIS}) is famously algebraic~\cite{BK}, and has asymptotics
$$
g(n) \, \sim \, \frac{ 2^{2/3}\ts\Ga(\frac13)}{3\ts\pi} \. 16^n \ts n^{-7/3}\ts.
$$

\subsection{}\label{ss:fin-two-ways}  There is more than one way
a sequence can be a diagonal of a rational function.  For example,
the Catalan numbers $C_n$ are the diagonals of
$$
\frac{1-x/y}{1-x-y} \quad \, \text{and} \quad
\frac{y\ts(1-2xy-2xy^2)}{1-x-2xy-xy^2}\,.
$$
The former follows from $\ts C_n= \binom{2n}{n}-\binom{2n}{n-1}$\ts, while
the second is given in~\cite{RY}.

\subsection{}\label{ss:fin-multisums}
There is a vast literature on binomials sums and multisums,
both classical and modern, see e.g.~\cite{PWZ,Rio}.
It was shown by Zeilberger~\cite{Zei} (see also~\cite{WZ}),
that under certain restrictions, the
resulting functions are $D$-finite, a crucial discovery
which paved a way to \emph{WZ~algorithm}, see~\cite{PWZ,WZ}.
A subclass of \emph{balanced multisums}, related but larger 
than~$\cBB$, was defined and studied in~\cite{Gar}. 
Note that the positivity is the not the only constraint we add.  
For example, balanced multisums in~\cite{Gar} easily contain 
Catalan numbers:  
$$\frac{(2n)! \. 1!}{n!\ts (n+1)!} \. = \. C_n\..
$$
We refer to~\cite[$\S$5.1]{B+} and~\cite{BLS} for the recent
investigations of binomial multisums which are diagonals
of rational functions, but without $\nn$-rationality
restriction. 

\subsection{}\label{ss:fin-gessel}
The class $\cR_1$ of $\nn$-rational GFs does not contain \emph{all}
of~$\nn[[x]]$ (Berstel,~1971); see~\cite{BR2,Ges} for some examples.
These are rare, however;  e.g.~Koutschan investigated
``about 60'' nonnegative rational GFs from~\cite{OEIS},
and found all of them to be in $\cR_1$, see~\cite[$\S$4.4]{Kou}.
In fact, there is a complete characterization of~$\cR_1$
by analytic means, via the Berstel~(1971) and Soittola~(1976) theorems.
We refer to~\cite{BR,SS} for these results and further references, and
to~\cite{Ges} for a friendly introduction.

Unfortunately, there is no such characterization of~$\cN$, nor
we expect there to be one, as singularities in higher dimensions
are most daunting~\cite{FS,PW}.  Even the most natural questions
remain open in that case (cf.~Conjecture~\ref{conj:growth}).
Here is one such question.

\begin{op}\label{op:gesel-comb-int}
Let $f \in \cF$ such that the corresponding {\rm GF} \. $F(x) \in \nn[[x]]$.
Does it follow that $f \in \cF_1$?
\end{op}

Personally, we favor a negative answer.  In~\cite{Ges}, Gessel asks whether
there are (nonnegative) rational GF which have a \emph{combinatorial
interpretation}, but are not $\nn$-rational.  Thus, a negative
answer to Problem~\ref{op:gesel-comb-int} would give a positive answer
to Gessel's question.\footnote{Christophe Reutenauer writes to us
that according the ``general metamathematical principle
that goes back to Sch\"utzenberger'' (see~\cite[p.~149]{BR2}),
the logic must be reversed:  a negative answer to
Gessel's question implies that the answer to
Problem~\ref{op:gesel-comb-int} must be positive.}
Of course, what's a combinatorial interpretation
is in the eye of the beholder; here we are implicitly assuming that our
irrational tilings or paths in graphs (see $\S$\ref{ss:enumR-cycles})
are \emph{nice} enough to pass this test (cf.~\cite{Ges}).  We plan
to revisit this problem in the future.

\subsection{}
There are over 200 different combinatorial interpretation of Catalan
numbers~\cite{Sta2}, some of them 1-dimensional such as the \emph{ballot
sequences}.  A quick review suggests that none of them can be
verified with a bounded memory read only TM.  For example, for the
ballot 0--1 sequences one must remember the running differences
$(\#0 \.- \.\#1)$, which can be large.  This gives some informal support
in favor of our Conjecture~\ref{conj:cat}.
Let us make following, highly speculative and
priceless claim.\footnote{Cf.\, {\tt http://tinyurl.com/mc3h8tn}.}

\begin{conj} \label{conj:cat-asym}
There is no tile counting function $f\in \cF$ which
is \emph{asymptotically Catalan:}
$$
f(n) \, \sim \, C_n \quad \ \text{as} \quad n \to \infty\ts.
$$
\end{conj}

We initially tried to disprove the conjecture.
Recall that by Lemma~\ref{l:cat3} and the technology in
Section~\ref{s:tech-apps}, it suffices to obtain
the constant $\ts\frac{\pi}{3\sqrt{3}}\ts$ as the product
of values of the hypergeometric functions given in Theorem~\ref{t:hypo}.
While $\ts\frac{1}{3\sqrt{3}}\ts$ is easy to obtain,
our hypergeometric sums seem too specialized to give value~$\pi$.
This is somewhat similar to the conjecture that $\frac{1}\pi$
is not a \emph{period}~\cite{KZ}.

We should mention here that it is rare when we can say anything
at all about the constant~$A$ in Conjecture~\ref{conj:growth}.
The constant in Corollary~\ref{c:hypo-transc-6} is an exception:
is known to be transcendental by the celebrated 1996 result of
Nesterenko on algebraic independence of $\pi$ and $\Ga(\frac14)$,
see~\cite{NP}.

\subsection{}\label{ss:fin-la}
The proof of Lemma~\ref{l:fin} uses a
generalization of a standard argument in combinatorial
linear algebra, for computing the number of rational tilings:
$$
F(x) \, = \, \sum_{n=0}^\infty \. f_T(n) \ts x^n \,
= \, \sum_{n=0}^\infty \. \left(M^n\right)_{00} \. x^n
\, = \,  \left(\frac{1}{1-M \ts x}\right)_{00} \, = \,
\frac{\det(1-M^{00} \ts x)}{\det(1-M \ts x)} \ts,
$$
where $M$ is the weighted adjacency matrix of~$G_T$.  It is thus
not surprising that we use a cycle decomposition argument somewhat
similar but more general than that in~\cite{CF,KP}.

In the same vein, the ``well-defined multiplicities'' argument in
the proof of Lemma~\ref{l:welld} is similar to the ``cycle popping''
argument in~\cite{Wil} (see also~\cite{GP1,Mar}).
The details are quite different, however.



\subsection{}\label{ss:fin-ord}
The values $\ts \text{ord}_p(C_n)$ \ts in Proposition~\ref{cat7}
were computed by Kummer (1852);
see~\cite{DS} for a recent combinatorial proof.

\vskip.6cm


\nin
{\bf Acknowledgements.} \ts We are grateful to Cyril Banderier,
Michael Drmota,  Ira Gessel, Marni Mishna, Cris Moore, Greta Panova,
Robin Pemantle, Christophe Reutenauer, Bruno Salvy, Andy Soffer,
and Richard Stanley for their interesting comments and generous
help with the references.  The first author was partially
supported by the University of California
Eugene V.~Cota-Robles Fellowship; the second author was partially
supported by the~NSF.

}



\vskip1.5cm

\begin{thebibliography}{BGPP}

\bibitem[BD]{BD}
C.~Banderier and M.~Drmota,
Formulae and asymptotics for coefficients of algebraic functions,
to appear in \emph{Combin.\/ Probab.\/ Comput.};
available at \. {\tt http://tinyurl.com/kny2xum}.

\bibitem[Bar]{Bar}
A.~Barvinok, \emph{Integer points in polyhedra}, European Mathematical Society,
Z\"{u}rich, 2008.

\bibitem[BL]{BL}
K.~P.~Benedetto and N.~A.~Loehr,
Tiling problems, automata, and tiling graphs,
\emph{Theoret. Comput. Sci.}~\textbf{407} (2008), 400--411.

\bibitem[BR1]{BR}
J.~Berstel and C.~Reutenauer,
\emph{Rational series and their languages},
Springer, Berlin, 1988. 

\bibitem[BR2]{BR2}
J.~Berstel and C.~Reutenauer,
\emph{Noncommutative rational series with applications},
Cambridge Univ.\ Press, Cambridge, 2011.

\bibitem[B+]{B+}
A.~Bostan, S.~Boukraa, G.~Christol, S.~Hassani and J.-M.~Maillard,
Ising $n$-fold integrals as diagonals of rational functions
and integrality of series expansions,
\emph{Journal of Physics A}~\textbf{46} (2013), 185202, 44~pp.;
extended version is available at \ts {\tt arXiv:1211.6031}, 100~pp.

\bibitem[BK]{BK}
A.~Bostan and M.~Kauers,
The complete generating function for Gessel walks is algebraic,
with an appendix by M.~van Hoeij,
\emph{Proc.\/ AMS.}~\textbf{138} (2010), 3063--3078.

\bibitem[BLS]{BLS}
A.~Bostan, P.~Lairez and B.~Salvy,
Repr\'{e}sentations int\'{e}grales des sommes binomiales,
talk slides (11 April, 2014, S\'{e}minaire Teich, Marseille);
available at \. {\tt http://tinyurl.com/n5l7kyt}.


\bibitem[BRS]{BRS}
A.~Bostan, K.~Raschel and B.~Salvy,
Non-$D$-finite excursions in the quarter plane,
{\em J.~Combin. Theory, Ser.~A}~\textbf{121} (2014), 45--63.

\bibitem[CCH]{CCH}
P.~Callahan, P.~Chinn and S.~Heubach,
Graphs of tilings, \emph{Congr. Numer.}~\textbf{183} (2006), 129--138.

\bibitem[CF]{CF}
P.~Cartier and D.~Foata,
\emph{Probl\`{e}mes combinatoires de commutation et r\'{e}arrangements},
Lecture Notes in Mathematics, No.~85, Springer, Berlin, 1969;
available at \. {\tt http://tinyurl.com/ntklmwv}

\bibitem[CLS]{CLS}
S.~Chen, N.~Li and S.~V.~Sam,
Generalized Ehrhart polynomials,
{\em Trans.\/ AMS}~\textbf{364} (2012), 551--569.

\bibitem[DS]{DS}
E.~Deutsch and B.~E.~Sagan,
Congruences for Catalan and Motzkin numbers and related sequences,
\emph{J.~Number Theory}~\textbf{117} (2006), 191--215.

\bibitem[DGS]{DGS}
B.~Dwork, G.~Gerotto and F.~J.~Sullivan,
An introduction to $G$-functions,
Princeton Univ.\/ Press, Princeton, NJ, 1994.

\bibitem[FGS]{FGS}
P.~Flajolet, S.~Gerhold and B.~Salvy,
On the non-holonomic character of logarithms, powers,
and the {$n$}th prime function,
{\em  Electron.\/ J.~Combin.}~\textbf{11} (2004/06), A2, 16~pp.

\bibitem[FS]{FS}
P.~Flajolet and R.~Sedgewick, \emph{Analytic Combinatorics},
Cambridge Univ.~Press, Cambridge, 2009.

\bibitem[Gar]{Gar}
S.~Garoufalidis,
{$G$}-functions and multisum versus holonomic sequences,
{\em Advances Math.}~\textbf{220} (2009), 1945--1955.

\bibitem[GaP]{GP}
S.~Garrabrant and I.~Pak,
Counting with Wang tiles, in preparation.

\bibitem[Ges1]{Ges1}
I.~Gessel, Two theorems of rational power series,
\emph{Utilitas Math.}~\textbf{19} (1981), 247--254.

\bibitem[Ges2]{Ges}
I.~Gessel, Rational functions with nonnegative integer coefficients,
in \emph{Proc. 50th S\'{e}m. Lotharingien de Combinatoire} (March 2003);
available at \. \url{http://tinyurl.com/krlbvvl}.

\bibitem[Gol]{Gol}
S.~W.~Golomb, \emph{Polyominoes}, Scribner, New York, 1965.

\bibitem[GoP]{GP1}
I.~Gorodezky and I.~Pak,
Generalized loop-erased random walks and approximate reachability,
\emph{Random Structures Algorithms}~\textbf{44} (2014), 201--223.


\bibitem[KM]{KM}
D.~A.~Klarner and S.~S.~Magliveras,
The number of tilings of a block with blocks,
\emph{European J.~Combin.}~\textbf{9} (1988), 317--330.

\bibitem[KP]{KP}
M. Konvalinka and I. Pak,
Non-commutative extensions of the MacMahon Master Theorem,
{\em Advances Math.}~\textbf{216} (2006), 29--61.

\bibitem[KZ]{KZ}
M.~Kontsevich and D.~Zagier,
\emph{Periods}, IHES preprint M/01/22 (May 2001), 38~pp.;
available at \. {\tt http://tinyurl.com/k7h7gvx}.

\bibitem[Korn]{Korn}
M.~R.~Korn,
\emph{Geometric and Algebraic properties of polyomino tilings},
Ph.D.~thesis, MIT, 2004; available at \.\url{http://dspace.mit.edu/handle/1721.1/16628}.

\bibitem[Kou]{Kou}
C.~Koutschan,
\emph{Regular languages and their generating functions: the inverse problem},
Master thesis, University of Erlangen-Nuremberg, Germany, 2005;
available at \. {\tt http://tinyurl.com/pj6l6gl}.

\bibitem[Mar]{Mar}
P.~Marchal, Loop-erased random walks and heaps of cycles,
Preprint  PMA-539, Univ.~Paris VI, 1999; available at \.
{\tt http://tinyurl.com/o54deou}.

\bibitem[MSV]{MSV}
D.~Merlini, R.~Sprugnoli and M.~Verri,
Strip tiling and regular grammars,
\emph{Theoret.\/ Comput.\/ Sci.}~\textbf{242} (2000) 109--124.

\bibitem[MM]{MM}
C.~Moore and S.~Mertens, \emph{The nature of computation},
Oxford Univ.\ Press, Oxford, UK, 2011.

\bibitem[NP]{NP}
Yu.~V.~Nesterenko and P.~Philippon (Eds.),
\emph{Introduction to algebraic independence theory},
Springer, Berlin, 2001.

\bibitem[Odl]{Odl}
A.~M.~Odlyzko, Asymptotic enumeration methods, in
\emph{Handbook of Combinatorics}, Vol.~2,
Elsevier, Amsterdam, 1995,  1063--1229.

\bibitem[Pak]{Pak}
I.~Pak,
\emph{Lectures on Discrete and Polyhedral Geometry},
monograph in preparation;
available at \. {\tt http://www.math.ucla.edu/\/\~\/pak/book.htm}.

\bibitem[PY]{PY}
I.~Pak and J.~Yang,
Tiling simply connected regions with rectangles,
\emph{J.~Combin. Theory, Ser.~A}~\textbf{120} (2013),
1804--1816.

\bibitem[PW]{PW}
R.~Pemantle and M.~C.~Wilson,
\emph{Analytic combinatorics in several variables},
Cambridge Univ.\/ Press, Cambridge, UK, 2013.

\bibitem[PWZ]{PWZ}
M.~Petkov\v{s}ek, H.~S.~Wilf and D.~Zeilberger,
$A=B$,  A~K~Peters, Wellesley, MA, 1996.

\bibitem[Pin]{Pin}
J.-E.~Pin,
Finite semigroups and recognizable languages: an introduction,
in \emph{Semigroups, formal languages and groups},
Kluwer, Dordrecht, 1995, 1--32.

\bibitem[Rio]{Rio}
J.~Riordan, \emph{Combinatorial identities},
John Wiley, New York, 1968.

\bibitem[RY]{RY}
E.~Rowland and R.~Yassawi,
Automatic congruences for diagonals of rational functions,
to appear in \emph{J.~Th\'{e}or.\/ Nombres Bordeaux};
available at \. {\tt arXiv:1310.8635}.


\bibitem[SS]{SS}
A.~Salomaa and M.~Soittola,
\emph{Automata-Theoretic Aspects of Formal Power Series},
Springer, New York, 1978.

\bibitem[Sch\"{u}]{Schu}
M.~P.~Sch\"{u}tzenberger,
On the definition of a family of automata,
\emph{Information and Control}~\textbf{4} (1961),
245--270

\bibitem[Sip]{Sip}
M.~Sipser, \emph{Introduction to the Theory of Computation}, PWS, 1997.

\bibitem[OEIS]{OEIS}
N.~J.~A.~Sloane,  \emph{The On-Line Encyclopedia of Integer Sequences}, {\tt  http://oeis.org}.

\bibitem[Sta1]{Sta}
R.~P.~Stanley, \emph{Enumerative combinatorics}, Vol.~1 and~2,
Cambridge Univ.\/ Press, Cambridge, UK, 1997 and 1999.

\bibitem[Sta2]{Sta2}
R.~P.~Stanley, \emph{Catalan Numbers}, monograph in preparation; an early draft titled
\emph{Catalan Addendum} is
available at \ts {\tt http://www-math.mit.edu/\/\~\/rstan/ec/catadd.pdf}.

\bibitem[Wang]{Wang}
H.~Wang, Games, logic and computers,
in \emph{Scientific American} (Nov.~1965), 98--106.

\bibitem[WZ]{WZ}
H.~Wilf and D.~Zeilberger, An algorithmic proof theory for hypergeometric
(ordinary and~$q$) multisum/integral identities,
\emph{Inventiones Math.}~\textbf{108} (1992), 575--633.

\bibitem[Wil]{Wil}
D.~B.~Wilson,
Generating random spanning trees more quickly than the cover time,
in {\em Proc.~28th STOC}, ACM, 1996, 296--303.

\bibitem[Zei]{Zei}
D.~Zeilberger,
Sister Celine's technique and its generalizations,
\emph{J.~Math.\/ Anal.\/ Appl.}~\textbf{85} (1982), 114--145.

\end{thebibliography}
\end{document}